\documentclass[pdflatex,sn-mathphys-num]{sn-jnl}

\usepackage{tikz,pgfplots}

\usepackage{graphicx}%
\usepackage{multirow}%
\usepackage{amsmath,amssymb,amsfonts}%
\usepackage{amsthm}%
\usepackage{mathrsfs}%
\usepackage[title]{appendix}%
\usepackage{xcolor}%
\usepackage{textcomp}%
\usepackage{manyfoot}%
\usepackage{booktabs}%
\usepackage{algorithm}%
\usepackage{algorithmicx}%
\usepackage{algpseudocode}%
\usepackage{listings}%
\usepackage{verbatim}
\usepackage{subfigure}
\usepackage{parskip}
\usepackage{hyperref}
\hypersetup{
  breaklinks=true,
  colorlinks=true,
  linkcolor=blue,
  citecolor=blue,
  filecolor=blue,
  urlcolor=blue,
  menucolor=blue
}
\usepackage{cleveref} 
\usepackage{cancel}   
\usepackage{bigints} 
\usepackage{upgreek}
\usepackage{enumitem}
\usepackage{textgreek}
\usepackage{amsfonts}
\usepackage{tikz}
\usetikzlibrary{arrows.meta, positioning}

\newcommand{\im}{\textup{i}}

\theoremstyle{thmstyleone}%
\newtheorem{theorem}{Theorem}[section]
\newtheorem{proposition}[theorem]{Proposition}%

\theoremstyle{thmstylethree}%
\newtheorem{definition}{Definition}%

\newtheorem{corollary}{Corollary}[section] 
\newtheorem{lemma}[theorem]{Lemma}

\begin{document}

\title[Equivalence of Synchronization States in the Kuramoto Flow: A Unified Framework]{Equivalence of Synchronization States in the Kuramoto Flow: A Unified Framework}


\author*[1]{\fnm{Ting-Yang} \sur{Hsiao}}\email{thsiao@sissa.it}

\author[2]{\fnm{Yun-Feng} \sur{Lo}}\email{ylo49@gatech.edu}

\author[3]{\fnm{Chengbin} \sur{Zhu}}\email{cz43@illinois.edu}

\affil*[1]{\orgdiv{Mathematics}, \orgname{Scuola Internazionale Superiore di Studi Avanzati (SISSA)}, \orgaddress{\street{via Bonomea 265}, \postcode{34136}, \state{Trieste}, \country{Italy}}}

\affil[2]{\orgdiv{School of Electrical and Computer Engineering}, \orgname{Georgia Institute of
Technology}, \orgaddress{\street{777 Atlantic Drive NW}, \city{Atlanta}, \postcode{30332}, \state{Georgia}, \country{USA}}}

\affil[3]{\orgdiv{Department of Mathematics}, \orgname{University of Illinois Urbana-Champaign}, \orgaddress{\street{1409 W Green St}, \city{Urbana}, \postcode{61801}, \state{Illinois}, \country{USA}}}


\abstract{
On the fiftieth anniversary of the Kuramoto model, synchronization remains a central paradigm for collective behavior in complex oscillator systems. While phase locking, frequency synchronization, and order-parameter coherence are widely used to characterize synchronization, their precise dynamical relationship has remained unclear, especially in finite-dimensional systems beyond mean-field limits. In this work, we show that, in the fully connected Kuramoto model, full phase locking, phase locking, frequency synchronization, and order-parameter synchronization are dynamically equivalent. Moreover, we demonstrate that the equivalence between phase locking and frequency synchronization is governed by a topology-independent mechanism and persists for generic symmetric coupling structures. These results provide a unified dynamical framework for synchronization and clarify the structural role of the Kuramoto order parameter.
}
\maketitle
\tableofcontents

\section{Introduction} \label{sec 1}
Synchronization is a fundamental phenomenon in nature, observed across a wide range of disciplines from biology and chemistry to physics and engineering. 
The mathematical formulation of synchronization through the Kuramoto model has become one of the most celebrated frameworks in nonlinear dynamics. 
Originally introduced by Yoshiki Kuramoto in the 1970s~\cite{kuramoto1975self}, the model provides a minimal description of collective synchronization in populations of coupled phase oscillators. 
Over the past five decades, this simple formulation has inspired extensive research across statistical physics, dynamical systems, and complex systems theory. The remarkable longevity of the Kuramoto model stems from its ability to capture emergent coherence in a wide variety of settings, including Josephson junction arrays~\cite{wiesenfeld1998frequency}, circadian rhythms~\cite{childs2008stability}, neural synchrony~\cite{breakspear2010generative}, power grids~\cite{dorfler2013synchronization}, spin-torque oscillators~\cite{flovik2016describing}, quantum synchronization models~\cite{lohe2009non, deville2019synchronization, bronski2020matrix}, algebraic generalizations~\cite{thumler2023synchrony, hsiao2025synchronization, hsiao2023synchronization}, and clustered oscillator systems~\cite{chen2024complete}. 
Across these diverse applications, the Kuramoto model has served as a paradigmatic reference point for understanding collective dynamical behavior.

From the perspective of physics, the Kuramoto model offers a paradigmatic example of a nonequilibrium system exhibiting spontaneous synchronization. 
A central concept in this framework is the Kuramoto order parameter, first introduced in~\cite{kuramoto1975self}, which provides a macroscopic measure of collective coherence. 
Defined as
\begin{align*}
Z(t) = R(t) e^{\im\Phi(t)} = \frac{1}{N} \sum_{j=1}^N e^{\im\theta_j(t)},
\end{align*}
the order parameter captures the degree of phase coherence in the oscillator population and plays a role analogous to magnetization in Landau's theory of phase transitions~\cite{landau1937theory}. The emergence of a nonzero magnitude $R(t)=|Z(t)|$ is often interpreted as a transition from incoherence to coherence and has become a standard diagnostic for synchronization in both theoretical and applied studies. 
Similar macroscopic coherence measures have been successfully employed in other physical systems exhibiting collective behavior, such as laser arrays~\cite{nair2021using}. 
As a result, the order parameter has become a cornerstone of synchronization theory and a primary tool for identifying synchronized states in complex oscillator systems.

As synchronization theory has developed over the past decades \cite{bronski2012fully, strogatz2000kuramoto, strogatz1993coupled,hong2011kuramoto,dorfler2011critical,dorfler2013synchronization, ha2021synchronization,ha2020asymptotic}, multiple criteria have been introduced to characterize collective behavior in Kuramoto-type systems. 
At the microscopic level, synchronization is often described in terms of phase locking, referring to bounded or convergent phase differences between oscillators. 
From a dynamical perspective, frequency synchronization emphasizes the alignment of asymptotic frequencies. 
At the macroscopic level, coherence is commonly quantified through the Kuramoto order parameter, which has become a standard indicator of synchronization in large-scale and numerical studies. Despite their widespread use, these notions of synchronization are frequently treated as interchangeable in the literature. 
In practice, conclusions drawn from one criterion—such as order-parameter coherence—are often assumed to imply others, such as phase or frequency synchronization. 
However, this equivalence is rarely stated explicitly and is typically justified only under restrictive assumptions, such as strong coupling, mean-field limits, or specific network structures.

In this work, we resolve this longstanding ambiguity by establishing a unified and fully nonlinear framework for synchronization in finite-dimensional Kuramoto systems. 
We show that, for generic symmetric coupling structures, full phase locking, phase locking, and frequency synchronization are dynamically equivalent. 
Moreover, in the fully connected Kuramoto model, we further prove that order-parameter synchronization is equivalent to these phase-based synchronization notions. Our results do not rely on mean-field limits, perturbative assumptions, or specific coupling topologies. 
Instead, they reveal a topology-independent dynamical mechanism underlying synchronization equivalence in Kuramoto flows. 
In this sense, the present work completes a structural aspect of synchronization theory that has remained implicit throughout the long history of the Kuramoto model.

\subsection{Kuramoto models}
The dynamics of interest in this paper are governed by the generalized first-order Kuramoto models:
\begin{align}  \label{GTK}
d_j \dot{\theta}_j = \omega_j + \frac{1}{N} \sum_{k=1}^N \lambda_{jk} \sin(\theta_k - \theta_j), \quad \text{for all } j \in \{1, \ldots, N\},
\end{align}
where $d_j>0$  is a given positive constant. Here, $\theta_j(t) \in \mathbb{R}$ denotes the phase of the $j$-th oscillator at time $t$ and $\omega_j \in \mathbb{R}$ is its natural frequency. The interaction strength between oscillators $j$ and $k$ is encoded by the coupling coefficient $\lambda_{jk}\in\mathbb{R}$. We work on a simple undirected graph $G=(V,E)$ with $V=\{1,\dots,N\}$ and weighted coupling matrix $\Lambda=[\lambda_{jk}]$ satisfying:
\begin{enumerate}[label=(\roman*),leftmargin=*]
\item \textbf{Weighted edges:} $\lambda_{jk}\in \mathbb{R}$ and $\lambda_{jj}=0$ (no self-loops).
\item \textbf{Symmetry (undirected):} $\lambda_{jk}=\lambda_{kj}$ for all $j,k$.
\item \textbf{Connectivity:} $G$ is connected.
\end{enumerate}
Throughout this paper, without loss of generality, we may assume that the effective natural frequencies satisfy the following normalization condition:
\begin{align} \label{omega constraint}
\omega_1 + \cdots + \omega_N = 0.
\end{align}
This can be achieved by introducing a change of variables:
\begin{align*}
\theta_j \rightarrow \theta_j - \frac{\omega_1 + \cdots + \omega_N}{d_1 + \cdots + d_N} t, \quad \omega_j \rightarrow \omega_j - \frac{\omega_1 + \cdots + \omega_N}{d_1 + \cdots + d_N} d_j,
\end{align*}
for all $j \in \{1, \ldots, N\}$. This transformation corresponds to moving into a co-rotating frame with a weighted average frequency. It preserves the dynamics of the phase differences and, consequently, all forms of synchronization considered in this work. It also simplifies the analysis by removing the uniform drift associated with the collective motion. In this co-rotating frame, we immediately obtain the following observation.\\

\begin{lemma} \label{warm up lemma}
    \begin{align} \label{first eq in}
        \lim_{t\rightarrow \infty} |\dot{\theta}_j(t)|=0, ~~\mbox{for all}~~ j\in\{1,\ldots,N\}
        .
    \end{align} if and only if 
    \begin{align} \label{first eq in warm up lemma}
        \lim\limits_{t\rightarrow\infty} |\dot{\theta}_j(t)-\dot{\theta}_k(t)|=0~~\mbox{for all}~j,k\in\{1,\ldots,N\}.
    \end{align}
\end{lemma}
\begin{proof}
    If $\theta(t)$ satisfies \eqref{first eq in}, then it is obvious that \eqref{first eq in warm up lemma} holds true. On the other hand, one finds that 
    \begin{align}\label{sum dot theta}
        d_1\dot{\theta}_1(t)+\ldots+d_N\dot{\theta}_N(t)=0,~~\mbox{for all}~t>0.
    \end{align}
    We notice that
    \begin{align} \label{delta inequality}
        \sum_{k=1}^N\left|d_k(\dot{\theta}_j-\dot{\theta}_k)\right|+\left|\sum_{k=1}^N d_k \dot{\theta}_k\right|\geq \left| \sum_{k=1}^N d_k \dot{\theta}_j\right|.
    \end{align}
    Inequality \eqref{delta inequality}, combined with \eqref{first eq in warm up lemma} and \eqref{sum dot theta}, shows that \eqref{first eq in}. The proof of Lemma \ref{warm up lemma} is complete.
\end{proof}

\color{black}
Finally, we summarize the parameter space of the generalized Kuramoto models \eqref{GTK}. Due to the rotational invariance of \eqref{GTK}, the natural frequencies satisfy the constraint \eqref{omega constraint}. Therefore, only $N-1$ frequencies are independent, and we identify the frequency vector with 
\begin{align*}
    \omega:=(\omega_1,\ldots,\omega_{N-1})\in\mathbb{R}^{N-1}.
\end{align*}
Moreover, the no self-loops coupling coefficients satisfy the symmetry condition 
\begin{align*}
    \lambda_{jk}=\lambda_{kj}\quad\text{for all} \quad j,k\in\{1,\ldots,N\} \quad \text{and}\quad \lambda_{jj}=0 \quad \text{for all} \quad j\in\{1,\ldots,N\}. 
\end{align*}
Hence the independent coupling parameters can be represented as 
\begin{align*}
    \Lambda:=\left[\lambda_{jk}\right]_{1\leq j<k\leq N}\in\mathbb{R}^{N(N-1)/2}
\end{align*}
Accordingly, the parameter space of the model is 
\begin{align} \label{P} 
    P:=\mathbb{R}^{N-1}\times\mathbb{R}^{N(N-1)/2}.
\end{align}
\color{black}

\section{Main results} \begin{figure}[htbp] 
    \centering
    \begin{tikzpicture}[thick, >=latex, scale=1.05, every node/.style={font=\bfseries}]

        \draw[black, thick] (0,0.2) ellipse (3.8cm and 3.9cm);
        \node at (0,4.4) {\large\bfseries Synchronization States};

        \node[circle, draw, minimum size=1.8cm] (FPLS) at (-2.5,1.5) {FPLS};
        \node[circle, draw, minimum size=1.8cm] (PLS) at (2.5,1.5) {PLS};
        \node[circle, draw, minimum size=1.8cm] (FSS) at (0,-2.3) {FSS};

        \node at (0,0) {\footnotesize (general topology)};

        \path[<->] (FPLS) edge[bend left=15] node[above=2pt] {\footnotesize Thm \ref{main 1} } (PLS);
        \path[<->] (FPLS) edge[bend right=15] node[left=2pt] {\footnotesize Thm \ref{main 1} } (FSS);
        \path[<->] (PLS) edge[bend left=15] node[right=2pt] {\footnotesize Thm \ref{main 1} } (FSS);

        \node[circle, draw, minimum size=1.8cm, align=center] (OPSS) at (-6.5,2.5) {OPSS};

        \draw[<->] (-3.7,1.2) -- (OPSS);
        \node at (-4.5,2.2) {\footnotesize Thm \ref{OP eqiv sync}};
        \node at (-5.2,1.4) {\footnotesize (all-to-all)};

    \end{tikzpicture}
\caption{
Logical structure of the equivalence relations among various synchronization states in the generalized Kuramoto flow. Each double arrow represents a bidirectional implication established in the corresponding theorem, under the assumptions described therein. The outer node \textbf{OPSS} (see, Definition~\ref{def 3}) corresponds to the order parameter synchronization state, whose relationship to the core synchronization states is proved under all-to-all topology (Theorems~\ref{OP eqiv sync}). \textbf{FPLS}, \textbf{PLS}, and \textbf{FSS} (see, Definition~\ref{def 2}) refer to the full phase-locked, phase-locked, and frequency synchronization states.
} \label{logical fig of main thm}
\end{figure}

We develop a rigorous, fully nonlinear framework that establishes the equivalence of four classical notions of synchronization—full phase-locking, phase-locking, frequency synchronization, order parameter synchronization—in generalized Kuramoto flows (see Fig.~\ref{logical fig of main thm}). The proofs are based on a compactness argument for the phase difference system, exploiting the periodic vector field structure and a finite-root condition to capture convergence mechanisms intrinsic to the system’s geometry. Notably, we rigorously justify the use of the Kuramoto order parameter as a valid indicator of synchronization, resolving a long-standing ambiguity in the field. In this paper, synchronization appears in several distinct dynamical states.
A full phase-locked state ($\mathrm{FPLS}$) refers to a regime in which each phase converges to a steady limiting configuration.
A phase-locked state ($\mathrm{PLS}$) denotes a state where all phase differences remain uniformly bounded for all time.
A frequency synchronization state ($\mathrm{FSS}$) corresponds to asymptotic alignment of instantaneous frequencies.
Finally, an order parameter synchronization state ($\mathrm{OPSS}$) describes a coherent macroscopic state whose limiting amplitude is dynamically sufficient to sustain phase locking. The defitions of ($\mathrm{FPLS}$), ($\mathrm{PLS}$), ($\mathrm{FSS}$), and ($\mathrm{OPSS}$) will be defined as follows. 

\subsection{Definitions of synchronization states}\label{def}
In this subsection, we provide a review of the concepts of phase synchronization, full phase-locked state, phase-locked state, and frequency synchronization state. Furthermore, we introduce the definitions of order parameter synchronization state.
\\
\begin{definition} \label{def 1}
A solution $\theta(t)$ is called phase synchronization state if for all $j,k\in\{1,\ldots,N\}$,
    \begin{align*}
        \lim_{t\rightarrow \infty} \left(\theta_j(t)-\theta_k(t)\right)=0. 
    \end{align*}
\end{definition}

\begin{definition} [Synchronization state]\label{def 2}

\begin{enumerate}
\item A solution $\theta(t)$ is called full phase-locked state $\mathrm{(FPLS)}$ if for all $j,k\in\{1,\ldots,N\}$,
    \begin{align*}
        \lim_{t\rightarrow\infty} (\theta_j(t)-\theta_k(t))=\theta^*_{jk},
    \end{align*}
where $\theta^*_{jk}$ are constants for all $j,k\in\{1,\ldots,N\}$. This is equivalent to the statement that $\theta(t)$ converges to a steady state $\theta^*$ satisfying the stationary equation:
\begin{align*}
    \omega_j+\frac{1}{N}\sum_{k=1}^N \lambda_{jk} \sin(\theta^*_{kj})=0,
\end{align*}
for all $j\in\{1,\ldots,N\}$.
\item  A solution $\theta(t)$ is called phase-locked state $\mathrm{(PLS)}$ if for all $j,k\in\{1,\ldots,N\}$, there exists a positive constant $L$ such that for all $t>0$,
    \begin{align*}
       |\theta_j(t)-\theta_k(t)|\leq L
       .
    \end{align*}
\item A solution $\theta(t)$ is called frequency synchronization state $\mathrm{(FSS)}$ if 
    \begin{align} \label{F S for G K}
        \lim_{t\rightarrow \infty} |\dot{\theta}_j(t)|=0, ~~\mbox{for all}~~ j\in\{1,\ldots,N\}
        .
    \end{align}
\end{enumerate} 
\end{definition}

\begin{definition} [OP synchronization] \label{def 3}
   Define order parameter for $\theta(t)$ as
   \begin{align} 
       Z(t):= R(t) e^{\im\Phi(t)}=\frac{1}{N}\sum_{j=1}^N e^{\im\theta_j(t)},
   \end{align}
   where $R(t)$ and $\Phi(t)$ are two real functions. In the fully connected homogeneous coupling case $\lambda_{jk}=\lambda\in\mathbb{R}\setminus\{0\}$, a solution $\theta(t)$ is called OP synchronization state $(\mathrm{OPSS})$ if there exist $Z^*\in\mathbb{C}$ such that
   \begin{align} \label{Phi limit}
       \lim_{t\rightarrow \infty} Z(t)=Z^*\quad\text{and}\quad \left|Z^*\right| \geq \frac{|\omega_j|}{|\lambda|}\quad\text{for all}\quad j\in\{1,\ldots,N\}.
   \end{align}
\end{definition}

We pause to remark that a necessary condition for the existence of a phase synchronization state is that all natural frequencies are identical. In other words, one must assume that  
\begin{align*}
    \omega_j -\omega_k \equiv 0, \quad \text{for all } j,k \in \{1, \dots, N\}.
\end{align*}  
Moreover, it is evident that if $\theta(t)$ is a full phase-locked state (see 1. in Definition~\ref{def 2}), then it is also a phase-locked state (see 2. in Definition~\ref{def 2}). Additionally, $R: \mathbb{R}^+ \to [0,1]$ is a continuous function. We also assume that $\Phi(0) \in [0,2\pi)$ and that $\Phi: \mathbb{R}^+ \to \mathbb{R}$ is a continuous function. Later, for all-to-all coupling, we will see that \eqref{Phi limit} is both necessary and sufficient (Theorem~\ref{OP eqiv sync}); see also Appendix~\ref{5_1}.

We define generic coefficients for the generalized first-order Kuramoto models \eqref{GTK}.
\begin{definition}[Generic coefficients] \label{GC}
Recall \eqref{P}. Let 
\begin{align}
    (\omega,\Lambda)\in P:=\mathbb{R}^{N-1}\times\mathbb{R}^{N(N-1)/2}.
\end{align}
denote the natural frequencies and coupling coefficients. We say the coefficients are generic if they belong to an open dense subset of the parameter space. 
\end{definition}

\subsection{Equivalence of synchronization states}
\begin{theorem}[Equivalence of synchronization states in the all-to-all Kuramoto model]\label{OP eqiv sync}
Consider the finite-dimensional Kuramoto model with all-to-all (fully connected $\lambda \in \mathbb{R}\setminus\{0\}$) coupling, that is, $\lambda_{jk} = \lambda \neq 0$ for all $j\neq k$:
\begin{align} \label{d classical Kuramoto}
d_j \dot{\theta}_j = \omega_j + \frac{\lambda}{N} \sum_{k=1}^N \sin(\theta_k - \theta_j), \quad \text{for all} \quad j \in \{1, \dots, N\}.
\end{align}
Then the following synchronization states are dynamically equivalent:
\begin{align*}
\mathrm{(FPLS)} \;\Longleftrightarrow\; \mathrm{(PLS)} \;\Longleftrightarrow\; 
\mathrm{(FSS)} \;\Longleftrightarrow\; \mathrm{(OPSS)}.
\end{align*}
In other words, the occurrence of any one of these synchronization states implies the others. In particular, setting $d_j=1$ for all $j$ shows that the classical Kuramoto model also has the above properties.
\end{theorem}

\begin{theorem}[Equivalence of synchronization states for generic coupling topologies] \label{main 1}
Consider the finite-dimensional Kuramoto model with arbitrary symmetric coupling topology:
\begin{align} \label{G Kuramoto}
d_j \dot{\theta}_j = \omega_j + \frac{1}{N} \sum_{k=1}^N \lambda_{jk} \sin(\theta_k - \theta_j), \quad \text{for all } j \in \{1, \ldots, N\}.
\end{align} 

For generic coefficients $(\omega,\Lambda)\in P$ (see, Definition~\ref{GC}), the following synchronization states
are dynamically equivalent:
\begin{align*}
\mathrm{(FPLS)} \;\Longleftrightarrow\; \mathrm{(PLS)} \;\Longleftrightarrow\; \mathrm{(FSS)}.
\end{align*}
That is, for generic $(\omega,\Lambda)\in P$, full phase locking, phase locking, and frequency synchronization states
describe the same asymptotic dynamical behavior of the Kuramoto flow, independently of the underlying network topology.
\end{theorem}

We end this subsection by explaining why a frequency synchronization state (FSS) necessarily implies a phase-locked state (PLS) in the all-to-all Kuramoto model.\\

\begin{lemma}\label{lemma classical K G}
Consider the classical Kuramoto model with $d_j=1$ and $\lambda_{jk}=\lambda\neq 0$:
\begin{align}\label{classical Kuramoto}
\dot{\theta}_j
= \omega_j + \frac{\lambda}{N}\sum_{k=1}^N \sin(\theta_k - \theta_j),
\quad j\in\{1,\ldots,N\}.
\end{align}
Let $\theta(t) = (\theta_1(t), \dots, \theta_N(t))$ be a solution of \eqref{classical Kuramoto}. 
If
\begin{align*}
\lim_{t \to \infty} \big( \dot{\theta}_i(t) - \dot{\theta}_j(t) \big) = 0
\quad \text{for all } i,j\in\{1,\ldots,N\},
\end{align*}
then
\begin{align*}
\sup_{t \ge 0} \big| \theta_i(t) - \theta_j(t) \big| < \infty
\quad \text{for all } i,j\in\{1,\ldots,N\}.
\end{align*}
\end{lemma}

\begin{proof}
The case where $(\omega_1,\ldots,\omega_N)=(0,\ldots,0)$ follows directly from \cite{watanabe1994constants}. Next, consider the case when $(\omega_1,\ldots,\omega_N)\neq (0,\ldots,0)$. We argue by contradiction. Suppose that $\theta(t)$ is a frequency synchronization state but not a phase-locked state. Then, there exists at least one pair $(\theta_j(t),\theta_k(t))$, for some $j,k\in\{1,\ldots,N\}$, such that
\[
\limsup_{t\to\infty} (\theta_j(t)-\theta_k(t)) = +\infty.
\]
Without loss of generality, assume that $\omega_j\neq \omega_k$ (if $\omega_j=\omega_k$, one can choose an index $l$ with $\omega_l\neq\omega_j$, and the unboundedness of some difference follows from either $\theta_j(t)-\theta_l(t)$ or $\theta_l(t)-\theta_k(t)$). For notational simplicity, we set $j=1$ and $k=N$.

Next, we choose $\theta_N(t)$ as the reference coordinate and define
\[
\psi_j(t)=\theta_j(t)-\theta_N(t), \quad j=1,\ldots,N,\quad \text{with } \psi_N(t)\equiv 0.
\]
Then \eqref{classical Kuramoto} can be rewritten as
\begin{equation}  \label{theta N reference}
\begin{aligned} 
\dot{\theta}_1-\dot{\theta}_N &= \omega_1-\omega_N + \frac{\lambda}{N}\sum_{l=1}^N \sin(\theta_l-\theta_1)-\frac{\lambda}{N}\sum_{l=1}^N \sin(\theta_l-\theta_N),\\[1mm]
&\;\;\vdots\\[1mm]
\dot{\theta}_{N-1}-\dot{\theta}_N &= \omega_{N-1}-\omega_N + \frac{\lambda}{N}\sum_{l=1}^N \sin(\theta_l-\theta_{N-1})-\frac{\lambda}{N}\sum_{l=1}^N \sin(\theta_l-\theta_N).
\end{aligned}    
\end{equation}
Expressed in terms of the new variables $\psi_j$, the system becomes
\begin{equation}  \label{psi N reference}
\begin{aligned} 
\dot{\psi}_1 &= \omega_1-\omega_N + \frac{\lambda}{N}\sum_{l=1}^{N} \sin(\psi_l-\psi_1)-\frac{\lambda}{N}\sum_{l=1}^{N-1} \sin(\psi_l),\\[1mm]
&\;\;\vdots\\[1mm]
\dot{\psi}_{N-1} &= \omega_{N-1}-\omega_N + \frac{\lambda}{N}\sum_{l=1}^{N} \sin(\psi_l-\psi_{N-1})-\frac{\lambda}{N}\sum_{l=1}^{N-1} \sin(\psi_l).
\end{aligned}    
\end{equation}
Since $\theta(t)$ is a frequency synchronization state, we have $\dot{\theta}_j(t)\to 0$ as $t\to\infty$ for all $j$, and hence $\dot{\psi}_j(t)\to 0$. However, by the construction of $\psi_1$, the assumption that $\theta(t)$ is not phase-locked implies that $\psi_1(t)$ is unbounded (in fact, $\psi_1(t)$ tends to infinity). In particular, there exists a time $T>0$ such that
\[
|\dot{\psi}_1(t)|<\frac{|\omega_1-\omega_N|}{2} \quad \text{for all } t>T.
\]
On the other hand, due to the unboundedness of $\psi_1(t)$, there exists a sequence $\{t_n\}$ with $t_n\to\infty$ such that
\[
\psi_1(t_n)=2\pi n \quad \text{for all } n\in\mathbb{N}.
\]
Evaluating \eqref{psi N reference} at these times leads to a contradiction with the bound on $|\dot{\psi}_1(t)|$. Thus, $\theta(t)$ must be phase-locked.

This completes the proof of Lemma~\ref{lemma classical K G}.
\end{proof}

\section{A unified framework for Kuramoto flows} \label{sec 4}

In this section, we provide rigorous proofs of our main results---Theorem~\ref{OP eqiv sync} and Theorem~\ref{main 1}---thereby establishing the equivalence of various synchronization definitions in the Kuramoto model. 

\subsection{Energy functional}

We first introduce the energy functional for the generalized Kuramoto model. For the classical Kuramoto model, the corresponding energy can be found in \cite{van1993lyapunov}.\\

\begin{lemma} \label{energy argument}
    Let $\theta(t)$ be a solution of \eqref{G Kuramoto}. If $\theta(t)$ is $\mathrm{(PLS)}$, then it is also  $\mathrm{(FSS)}$.
\end{lemma}

\begin{proof}
    Multiplying $\dot{\theta}_j$ in \eqref{G Kuramoto}, summing over $j=1,\ldots,N$ and integrating the term over $(0,t)$, we obtain
    \begin{equation} \label{energy function}
    \begin{aligned}
        \int_0^t\sum_{j=1}^{N}d_j\dot{\theta}^2_j(s) \mathrm{d}s=&\sum_{j=1}^N \omega_j\left(\theta_j(t)-\theta_j(0)\right)\\
        &+\int_0^t \frac{1}{N} \sum_{j=1}^N\sum_{k=1}^N\lambda_{jk}\sin(\theta_k(s)-\theta_j(s))\dot{\theta}_j(s) \mathrm{d}s.
    \end{aligned}
    \end{equation}
    Next, we observe that the term
    \begin{align} \label{Lambda}
        \sum_{j=1}^N \omega_j\left(\theta_j(t)-\theta_j(0)\right)=\sum_{j=2}^N \omega_j(\theta_j(t)-\theta_1(t))-\sum_{j=1}^N \omega_j\theta_j(0), 
    \end{align}
    remains bounded uniformly in $t$. Moreover, using the symmetry $\lambda_{jk}=\lambda_{kj}$, one obtain
    \begin{equation} \label{H}
        \begin{aligned}
            \int_0^t \sum_{j=1}^N\sum_{k=1}^N\lambda_{jk}\sin(\theta_k(s)-\theta_j(s))\dot{\theta}_j(s) \mathrm{d}s =\left.\sum_{j<k} \lambda_{jk}\cos(\theta_k(s)-\theta_j(s))\right|_0^t,
        \end{aligned}
    \end{equation}
    which is also uniformly bounded in $t$. Hence, from
  \eqref{energy function}, combined with \eqref{Lambda} and \eqref{H}, we deduce that
    \begin{align} \label{energy is bdd}0\leq\int_0^t\sum_{j=1}^{N}d_j\dot{\theta}^2_j(s) \mathrm{d}s\leq \int_0^\infty\sum_{j=1}^{N}d_j\dot{\theta}^2_j(s) \mathrm{d}s \leq C,
    \end{align}
    for some constant $C>0$ independent of $t$.
    
 Next, we observe that 
\begin{equation} \label{theta dot is bdd 1}
|\dot{\theta}_j(t)|\leq \left|\frac{\omega_j}{d_j}\right|+\frac{1}{N d_j}\sum_{k=1}^N|\lambda_{jk}|, \quad \text{for all } j.
\end{equation}
It follows from \eqref{theta dot is bdd 1} that $\dot{\theta}_j(t)$ is uniformly bounded and, by differentiating \eqref{G Kuramoto}, so is $\ddot{\theta}_j(t)$. Consequently, the function $\sum_{j=1}^{N}d_j\dot{\theta}_j^2(t)$ is uniformly continuous on $t>0$.

By \eqref{energy is bdd} and the uniform continuity of $\sum_{j=1}^{N}d_j\dot{\theta}_j^2(t)$, we obtain
\[
\lim_{t\to\infty}\sum_{j=1}^{N}d_j\dot{\theta}_j^2(t)=0.
\]
Thus, we conclude that $\dot{\theta}_j(t)\to 0$ for all $j$, meaning that $\theta(t)$ is a frequency synchronization state. This completes the proof.
\end{proof}

\subsection{Generic finite-root property via parametric transversality}

This section is concerned with the application of the Parametric Transversality Theorem to the finite-root property of the generalized Kuramoto models. The key theorem (Theorem~\ref{PTT}) is due to \cite[Sard]{sard1942measure}. A complete discussion of the theorem, its proof and generalizations can be found in the textbook \cite[Theorem~6.35]{lee2003smooth}. To the best of the authors' knowledge, the finite-root property of the Kuramoto equilibrium equations has not been rigorously established in the existing literature. 
Although the finiteness of equilibria has been mentioned conceptually in earlier works (see, e.g., \cite{baillieul1982geometric}), a complete analytical proof appears to be absent. 
The present work provides a rigorous justification of this property using a parametric transversality argument.\\

\begin{theorem} [Parametric Transversality Theorem] \label{PTT}
Let $\mathsf{N},\mathsf{M}$ be smooth manifolds (over $\mathbb{R}$), and let $\mathsf{X}\subset \mathsf{M}$ be an embedded submanifold. 
Let $\{\mathsf{F}_p : \mathsf{N} \to \mathsf{M}\}_{p\in \mathsf{P}}$ be a smooth family of maps parametrized by a smooth manifold $\mathsf{P}$, 
and define
\begin{align*}
\mathsf{F}:\mathsf{N}\times \mathsf{P} \to \mathsf{M}, \qquad \mathsf{F}(x,p)=\mathsf{F}_p(x).
\end{align*}
If $\mathsf{F}$ is transverse to $\mathsf{X}$, that is, for every $(x,p)\in \mathsf{N}\times \mathsf{P}$ with 
$\mathsf{F}(x,p)\in \mathsf{X}$ we have
\begin{align*}
D_{(x,p)}\mathsf{F}\left(T_{(x,p)}(\mathsf{N}\times \mathsf{P})\right)
+ T_{\mathsf{F}(x,p)}\mathsf{X}
=
T_{\mathsf{F}(x,p)}\mathsf{M},
\end{align*}
then for almost every $p\in \mathsf{P}$, the map $\mathsf{F}_p:\mathsf{N}\to \mathsf{M}$ is transverse to $\mathsf{X}$.
\end{theorem}

Let us choose 
\begin{align*}
    \mathsf{N}=\mathbb{R}^{N-1},\quad \mathsf{P}=\mathbb{R}^{N-1+N(N-1)/2},\quad \mathsf{M}=\mathbb{R}^{N-1},\quad \mathsf{X}=\{0\},
\end{align*}
and denote
\begin{align*}
    \psi:=(\psi_1,\ldots,\psi_{N-1})\in \mathsf{N},\quad (\omega,\Lambda)\in \mathsf{P}.
\end{align*}
Let $\{\mathsf{F}(\cdot;\omega,\Lambda):\mathsf{N}\to \mathsf{M}\}_{(\omega,\Lambda)\in \mathsf{P}}$ be a smooth family of maps parametrized by a smooth manifold $\mathsf{P}$, and define
\begin{align}
    \mathsf{F}:\mathsf{N}\times \mathsf{P}\to \mathsf{M}, \quad \mathsf{F}(\psi,\omega,\Lambda)=\mathsf{F}(\psi;\omega,\Lambda),
\end{align}
where, for $j\in\{1,\ldots,N-1\}$,
\begin{align} \label{Fj}
    \mathsf{F}_j(\psi,\omega,\Lambda):=\omega_j - \omega_N
+ \frac{1}{N}\sum_{l=1}^{N}\lambda_{jl}\sin(\psi_l-\psi_j)
- \frac{1}{N}\sum_{l=1}^{N-1}\lambda_{Nl}\sin(\psi_l),
\end{align}
and
\begin{align}
    \psi_N:=0, \quad \omega_N:=-\omega_1-\omega_2-\ldots-\omega_{N-1}.
\end{align}

Let us show $\mathsf{F}$ is transverse to $X=\{0\}$. A straightforward calculation reveals that
\begin{align}
    D_\omega\mathsf{F}=I_{N-1}+J_{N-1}=\begin{pmatrix}
1 & 0 & 0 & \cdots & 0 \\
0 & 1 & 0 & \cdots & 0 \\
0 & 0 & 1 & \cdots & 0 \\
\vdots & \vdots & \vdots & \ddots & \vdots \\
0 & 0 & 0 & \cdots & 1
\end{pmatrix}+\begin{pmatrix}
1 & 1 & 1 & \cdots & 1 \\
1 & 1 & 1 & \cdots & 1 \\
1 & 1 & 1 & \cdots & 1 \\
\vdots & \vdots & \vdots & \ddots & \vdots \\
1 & 1 & 1 & \cdots & 1
\end{pmatrix}.
\end{align}
Since $J_{N-1}$ has rank one, $I_{N-1}+J_{N-1}$ is invertible, this implies
\begin{align}
   \mathit{rank}\, D_{(\psi,\omega,\Lambda)}F
=\mathit{rank}\,
\left[
\begin{array}{c|c|c}
D_\psi \mathsf{F} & D_\omega \mathsf{F}& D_\Lambda \mathsf{F}
\end{array}
\right]=N-1,
\end{align}
and, hence,
\begin{align*}
    D_{(\psi,\omega,\Lambda)}\mathsf{F}\left(T_{(\psi,\omega,\Lambda)}(\mathsf{N}\times \mathsf{P})\right)
+ T_{0}\{0\}
= \mathbb{R}^{N-1}+\{0\}=\mathbb{R}^{N-1}=
T_{0}\mathsf{M}.
\end{align*}
By Theorem~\ref{PTT}, for almost every $(\omega,\Lambda)\in \mathsf{P}$, the map $\mathsf{F}_{\omega,\Lambda}(\cdot):=\mathsf{F}(\cdot;\omega,\Lambda):\mathsf{N}\to \mathsf{M}$ is transverse to $\{0\}$. This transversality condition for the map $\mathsf{F}_{\omega,\Lambda}$, by definition, means that for each $\psi\in\mathsf{N}$ such that $ \mathsf{F}_{\omega,\Lambda}(\psi)\in\{0\}$, we have
\begin{align}
    D_{\psi} \mathsf{F}_{\omega,\Lambda}\left(T_{\psi}\mathsf{N}\right)=T_0\mathsf{M}=\mathbb{R}^{N-1},
\end{align}
which implies that $D_{\psi} \mathsf{F}_{\omega,\Lambda}$ is surjective, hence invertible. By the Inverse Function Theorem for Manifolds~\cite[Theorem~4.5]{lee2003smooth}, around each $\psi\in\left(\mathsf{F}_{\omega,\Lambda}\right)^{-1}(\{0\})\subset\mathsf{N}$ there exists a connected neighborhood $\mathsf{U}\subset\mathsf{N}$, and a corresponding connected neighborhood $\mathsf{V}\subset\mathsf{M}$ around $\mathsf{F}_{\omega,\Lambda}(\psi)=0$ such that $\mathsf{F}_{\omega,\Lambda}\vert_{\mathsf{U}}:\mathsf{U}\to\mathsf{V}$ is a diffeomorphsim, which is invertible. This invertible function is injective, so locally $\psi$ is the only zero of $\mathsf{F}_{\omega,\Lambda}$. That is, each zero of $\mathsf{F}_{\omega,\Lambda}$ is \emph{isolated}. Recall that this holds for almost every parameter $(\omega,\Lambda)$, i.e., the parameters that this does not hold has Lebesgue measure zero. 

Fix a parameter combination $(\omega,\Lambda)$ such that each zero of $F_{\omega,\Lambda}$ is isolated, and restrict our attention to a length-$2\pi$ box in $\mathbb{R}^{N-1}$. We claim that the zeros of $F_{\omega,\Lambda}$ can only be finitely many in this box, by contradiction. If there were infinite zeros in this compact box, then there exists a sequence $(\psi^{(j)})_{j\in\mathbb{N}}$ of zeros of $F_{\omega,\Lambda}$ in the box. The box is compact, hence sequentially compact, so after potentially passing to a subsequence, $(\psi^{(j)})_{j\in\mathbb{N}}$ converges to some $\psi^\star$ in the box. By continuity of $F_{\omega,\Lambda}$, we know that $\psi^\star$ is also a zero of $F_{\omega,\Lambda}$. This is a contradiction, since $\psi^{(j)}\to\psi^\star$ as $j\to\infty$ but by assumption, each zero of $F_{\omega,\Lambda}$ is isolated. We have shown that in this box, the zeros of $F_{\omega,\Lambda}$ are finitely many. Conversely, if there are finitely many zeros of $F_{\omega,\Lambda}$ in the box, each of them is easily seen to be isolated. Hence, in the compact box, $F_{\omega,\Lambda}$ having isolated zeros is equivalent to it having finitely many zeros.

Now we are ready to present the finite-root property of the generalized Kuramoto models.\\

\begin{proposition}\label{finite root proposition}
For generic coefficients $(\omega,\Lambda)\in 
\mathsf{P}=\mathbb{R}^{N-1}\times \mathbb{R}^{N(N-1)/2}$,
the system
\begin{align*}
\mathsf{F}_{\omega,\Lambda}(\psi_1,\ldots,\psi_{N-1})=0
\end{align*}
has only finitely many roots in the fundamental domain $U=[0,2\pi]^{N-1}$.
Consequently, for every $l\in\mathbb{Z}^{N-1}$, the system has only finitely many roots in the translated box
\begin{align*}
U_l:=U+2\pi l,
\end{align*}
and the number of roots in $U_l$ is the same as that in $U$.
\end{proposition}
\begin{proof}

Define
\[
\mathcal{T}
:=
\left\{
(\omega,\Lambda)\in P :
\mathsf{F}_{\omega,\Lambda}\,\text{is transverse to}\, \{0\}
\text{ on } U
\right\}.
\]

By the Theorem~\ref{PTT}, the complement
$P\setminus\mathcal{T}$ has Lebesgue measure zero. In particular,
$\mathcal{T}$ is dense in $P$. We next prove that $\mathcal{T}$ is open. Let $(\omega_0,\Lambda_0)\in\mathcal{T}$.

\textbf{Case 1.} Suppose
\[
\mathsf{F}_{\omega_0,\Lambda_0}^{-1}(0)\cap U = \varnothing.
\]

Since $U$ is compact and $\mathsf{F}_{\omega_0,\Lambda_0}$ is continuous,
the function $\psi\mapsto |\mathsf{F}_{\omega_0,\Lambda_0}(\psi)|$ attains
a positive minimum on $U$. Hence there exists $c>0$ such that
\[
\min_{\psi\in U} |\mathsf{F}_{\omega_0,\Lambda_0}(\psi)| = c > 0.
\]

By continuity of $\mathsf{F}(\psi,\omega,\Lambda)$ with respect to
$(\psi,\omega,\Lambda)$, for all $(\omega,\Lambda)$ sufficiently
close to $(\omega_0,\Lambda_0)$ we have
\[
\sup_{\psi\in U}
|\mathsf{F}_{\omega,\Lambda}(\psi)-\mathsf{F}_{\omega_0,\Lambda_0}(\psi)|
< \frac{c}{2}.
\]
Therefore for every $\psi\in U$,
\[
|\mathsf{F}_{\omega,\Lambda}(\psi)|
\ge
|\mathsf{F}_{\omega_0,\Lambda_0}(\psi)|-\frac{c}{2}
\ge
\frac{c}{2}
>0.
\]

Hence $\mathsf{F}_{\omega,\Lambda}$ has no zero in $U$.
Thus $\mathsf{F}_{\omega,\Lambda}$ is vacuously transverse to $\{0\}$ on $U$,
so $(\omega,\Lambda)\in\mathcal{T}$.

\textbf{Case 2.} Suppose
\[
\mathsf{F}_{\omega_0,\Lambda_0}^{-1}(0)\cap U
=
\{\psi^1,\dots,\psi^m\}.
\]

Since $(\omega_0,\Lambda_0)\in\mathcal{T}$,
each zero is nondegenerate, i.e.
\[
\det D_\psi \mathsf{F}_{\omega_0,\Lambda_0}(\psi^j)\neq0,
\qquad j=1,\dots,m.
\]

By the implicit function theorem, for each $j$
there exist an open neighborhood $W_j\subset U$ of $\psi^j$, an open neighborhood $\mathcal{O}_j\subset P$ of $(\omega_0,\Lambda_0)$, and a smooth map $\psi^j(\omega,\Lambda)$, such that for every $(\omega,\Lambda)\in\mathcal{O}_j$ the equation
\[
\mathsf{F}_{\omega,\Lambda}(\psi)=0
\]
has exactly one solution in $W_j$, namely
\[
\psi=\psi^j(\omega,\Lambda),
\]
and this solution satisfies
\[
\det D_\psi \mathsf{F}_{\omega,\Lambda}(\psi^j(\omega,\Lambda))\neq0.
\]
Shrinking the neighborhoods $W_j$ if necessary, we may assume
their closures $\overline{W_j}$ are pairwise disjoint. Define
\[
K := U \setminus \bigcup_{j=1}^m W_j.
\]

Then $K$ is compact and contains no zero of
$\mathsf{F}_{\omega_0,\Lambda_0}$. Hence
\[
c :=
\min_{\psi\in K}
|\mathsf{F}_{\omega_0,\Lambda_0}(\psi)|
>0.
\]
By continuity of $\mathsf{F}$,
there exists a neighborhood $\mathcal{O}_K$ of
$(\omega_0,\Lambda_0)$ such that
for every $(\omega,\Lambda)\in\mathcal{O}_K$,
\[
\sup_{\psi\in K}
|\mathsf{F}_{\omega,\Lambda}(\psi)-\mathsf{F}_{\omega_0,\Lambda_0}(\psi)|
<
\frac{c}{2}.
\]
Consequently, for all $\psi\in K$,
\[
|\mathsf{F}_{\omega,\Lambda}(\psi)|
\ge
|\mathsf{F}_{\omega_0,\Lambda_0}(\psi)|-\frac{c}{2}
\ge
\frac{c}{2}
>0.
\]
Thus $\mathsf{F}_{\omega,\Lambda}$ has no zero in $K$. Therefore every zero of $\mathsf{F}_{\omega,\Lambda}$ in $U$
must lie in one of the sets $W_j$.
Since each $W_j$ contains exactly one zero,
all zeros remain nondegenerate. Hence $\mathsf{F}_{\omega,\Lambda}$ is transverse to $\{0\}$ on $U$,
and therefore $(\omega,\Lambda)\in\mathcal{T}$.
This proves that $\mathcal{T}$ is open. Finally, fix $(\omega,\Lambda)\in\mathcal{T}$.
Every zero of $\mathsf{F}_{\omega,\Lambda}$ in $U$ is nondegenerate,
hence isolated. Since $U$ is compact, the set $\mathsf{F}_{\omega,\Lambda}^{-1}(0)\cap U$ must be finite.

The periodicity
\[
\mathsf{F}_{\omega,\Lambda}(\psi+2\pi l)
=
\mathsf{F}_{\omega,\Lambda}(\psi),
\qquad l\in\mathbb Z^{N-1},
\]
implies that translation by $2\pi l$ gives a bijection
between roots in $U$ and roots in $U_l$.
Hence each translated box $U_l$ contains finitely many
roots, and the number of roots is the same as that in $U$.
\end{proof}

Therefore, we are now in a position to prove Theorem~\ref{main 1}.
By combining Theorem~\ref{PTT} with Proposition~\ref{finite root proposition}, we ensure that the finite-root property holds for parameters $(\omega,\Lambda)$ in the generic sense (see, Definition~\ref{GC}).
Consequently, the set of parameters in generalized Kuramoto models for which the equivalence of synchronization notions may fail is of Lebesgue measure zero.

\color{black}
\subsection{Proof of Theorem \ref{main 1}}
\begin{proof}
$\mathrm{(PLS)}$ implies $\mathrm{(FSS)}$: This result follows directly from Lemma \ref{energy argument}.

$\mathrm{(FSS)}$ implies $\mathrm{(PLS)}$: We proceed again by contradiction. Suppose otherwise that $\theta(t)$ is not a phase-locked state. There exists at least one pair $(\theta_l(t),\theta_k(t))$, for some $l,k\in\{1,\ldots,N\}$ such that 
 \begin{align} \label{psi_1 is not bdd}
     \limsup\limits_{t\rightarrow \infty} (\theta_l(t)-\theta_k(t))=+\infty.
 \end{align}
 Without loss of generality, we let $l=1$ and $k=N$. Let us rewrite the system equations \eqref{G Kuramoto} as
\begin{equation}  \label{psi N reference G K}
\begin{aligned} 
    d_j\dot{\theta}_j-d_N\dot{\theta}_N=\mathsf{F}_j(\psi_1,\ldots,\psi_{N-1};\omega,\Lambda)
\end{aligned}    
\end{equation}
where, for $j\in\{1,\ldots,N-1\}$,
\begin{align*}
\psi_j(t)=\theta_j(t)-\theta_N(t),    
\end{align*}
$\psi_N(t)\equiv 0$, and 
\begin{align*}
\mathsf{F}_j(\psi;\omega,\Lambda)=\omega_j-\omega_N + \frac{1}{N} \sum_{l=1}^{N} \lambda_{jl}\sin(\psi_l-\psi_j)-\frac{1}{N}\sum_{l=1}^{N-1} \lambda_{Nl}\sin(\psi_l).
\end{align*}
It is clear that $f_j$ is periodic and satisfies 
\begin{align} \label{periodic f}
\mathsf{F}_j(\psi_1,\ldots,\psi_k,\ldots,\psi_{N-1})=\mathsf{F}_j(\psi_1,\ldots,\psi_k+2\pi,\ldots,\psi_{N-1}),
\end{align}
for all $j,k\in\{1,\ldots,N-1\}$.
Partition $\mathbb{R}^{N-1}$ into a grid of boxes $\{U_n\}_{n=1}^\infty$, each of which is a closed hypercube with side length $2\pi$. By Proposition~\ref{finite root proposition}, we observe that the system $\mathsf{F}_j(\psi;\omega,\Lambda)=0,~\forall j\in\{1,\ldots,N-1\}$ has finite roots in each box. Using these roots $\{r_n^1,r_n^2,\ldots,r_n^m\}\subset U_n$ as centers, construct open balls $\{B_{n1}^\epsilon,B^\epsilon_{n2},\ldots,B^\epsilon_{nm}\}\subset\mathbb{R}^{N-1}$ such that the balls are disjoint. There exists a $\delta>0$ such that $|\mathsf{F}_1(\psi;\omega,\Lambda)|>\delta$ on the closed and bounded set $U_n\setminus\cup_{j=1}^m \{B^\epsilon_{nj}\}$. Due to periodicity, we can repeat the exact same process to obtain the same roots, open balls and have $|\mathsf{F}_1(\psi;\omega,\Lambda)|>\delta$ on the closed and bounded sets $U_n\setminus\cup_{j=1}^m \{B^\epsilon_{nj}\}$, for all $n=1,\ldots$. 

On the other hand, the trajectory $(\psi_1(t),\ldots,\psi_{N-1}(t))$ lies in $ \mathbb{R}^{N-1}$ and, due to \eqref{psi_1 is not bdd}, its first component is not bounded. This implies that the trajectory passes through infinitely many distinct boxes, and $|\mathsf{F}_1|>\delta$ occurs infinitely often. In other words, for any natural number $\tilde{n}\in\mathbb{N}$, there always exists a time $\tilde{n}$ such that $|\mathsf{F}_1(\psi(t);\omega,\Lambda)|>\delta$. This contradicts our assumption that $\theta(t)$ is a frequency synchronization state, since the left-hand side of \eqref{psi N reference G K} tends to zero as $t$ tends to infinity. 

To sum up, $\theta(t)$ must be a phase-locked state. 

$\mathrm{(PLS)}$ implies $\mathrm{(FPLS)}$: We follow previous argument. Given sufficiently small $\epsilon>0$, we construct disjoint open balls $B_{nl}^\epsilon$, centered at $r_n^l$ for all $n\in\mathbb{N}$ and $l\in\{1,\ldots,m\}$. From the previous proof, we know that $\theta(t)$ is a phase-locked state or, equivalently, a frequency synchronization state. Following from similar argument, we know $|\mathsf{F}_j(\psi;\omega,\Lambda)|$ is bounded below by a positive constant, for all $j\in\{1,\ldots,N-1\}$ and for all $(\psi_1,\ldots,\psi_{N-1})\in \mathbb{R}^{N-1}\setminus \cup_{n=1}^{\infty}\cup_{l=1}^{m} \{B^\epsilon_{nl}\}$. Therefore, there exists a sufficiently large $T>0$, such that the trajectory lies in $B^\epsilon_{nl}$ for some $n\in\mathbb{N}$ and for some $l\in\{1,\ldots,N-1\}$ when the time is sufficiently long; in other words, $(\psi_1(t),\ldots,\psi_{N-1}(t))\in B^\epsilon_{nl}$, for all $t>T$.

Using a similar approach, we can find a smaller $\epsilon_1$ and demonstrate that the trajectory will eventually fall into a smaller $\epsilon_1$-ball. Continuing this argument indefinitely, we know that the trajectory will ultimately approach a certain $r^j_n$ (the equilibrium point). Thus, we have proven that if $\theta(t)$ is a phase-locked state, then $\theta(t)$ is also a full phase-locked state.

$\mathrm{(FPLS)}$ implies $\mathrm{(PLS)}$ This result follows directly from Definition \ref{def 2}.

The proof of Theorem \ref{main 1} is complete.
\end{proof}
\subsection{Proof of Theorem \ref{OP eqiv sync}}

\begin{proof}
In the case of the all-to-all coupling topology, the finite-root property has already been established in \cite{ha2016finiteness}. Therefore, by Theorem~\ref{main 1}, all synchronization notions in Definition~\ref{def 2} are equivalent. Suppose that $\theta(t)$ is a synchronization state (Definition \ref{def 2}). There exist constants $\theta_j^*\in\mathbb{R}$, for all $j\in\{1,\ldots,N\}$ such that 
\begin{align*}
    \lim_{t\rightarrow\infty} \theta_j(t)=\theta_j^* ~~\mbox{for}~~\mbox{all}~~j\in\{1,\ldots,N\}.
\end{align*}
Recalling Definition \ref{def 3}, we obtain $Z(t)$ tends to be a constant as $t$ tends to infinity. This shows that there exists $0<R^*\leq 1$ and $\Phi^*\in\mathbb{R}$ such that 
\begin{align*}
    \lim_{t\rightarrow \infty} R(t)=R^*~~\mbox{and}~~\lim_{t\rightarrow \infty} \Phi(t)=\Phi^*.
\end{align*}

Suppose that there exists $k$ such that 
\begin{align*}
    |\omega_k|>R(t)|\lambda|~~\mbox{when}~t~\mbox{is sufficiently large}.
\end{align*}
This demonstrates that $|\dot{\theta}_k|$ is bounded below by a nonzero constant, so $\theta(t)$ is not a frequency synchronization state. The contradiction asserts the proof.

On the other hand, assume that $\theta(t)$ is an OP synchronization state, or equivalently satisfies \eqref{Phi limit}. Suppose that $\theta(t)$ is not a synchronization state (Definition \ref{def 2}). Then, there exists at least one $\ell\in\{1,\ldots,N\}$ such that $|\theta_\ell|$ is unbounded. Recall that we assume $\max_{j\in\{1,\ldots,N\}}|\omega_j|=|\omega_k|$. In the following discussion, to simplify the exposition, we assume 
$\lambda>0$ without loss of generality. Suppose that $R^*>|\omega_k|/\lambda$. Then we have $R^*>|\omega_\ell|/\lambda$. Then there exists $0<\epsilon \ll 1$ such that 
\begin{align*}
    |\omega_\ell|<\lambda (R^*-\epsilon)\sin
    \left(\frac{\pi}{2}-\epsilon\right).
\end{align*}
By \eqref{Phi limit}, for this $\epsilon$, there exists $T>0$ such that 
\begin{align*} 
    R(t)>R^*-\epsilon~~\mbox{and}~~ |\Phi(t)-\Phi^*|<\epsilon~~\mbox{for}~~\mbox{all}~~t>T
    .
\end{align*}

Since $|\theta_\ell|$ is unbounded, there exists the first moment $t^*_1>T$ such that $\theta_\ell(t^*_1)=\Phi^*+\pi/2+2n_1\pi$ for some $n_1\in\mathbb{N}^+$, which implies $\dot{\theta}_\ell(t^*_1)\geq 0$, or there exists the first moment $t^*_2>T$ such that $\theta_\ell(t^*_2)=\Phi^*-\pi/2-2n_2\pi$ for some $n_2\in\mathbb{N}^+$, which implies $\dot{\theta}_\ell(t^*_2)\leq 0$. For the first case, we obtain, at $t=t^*_1$,
\begin{align*}
    d_\ell \dot{\theta}_\ell
    =
    \omega_\ell + \lambda R(t_1^*) \sin\left( \Phi(t_1^*) - \Phi^* - \frac{\pi}{2} \right)
    <
    \omega_\ell-\lambda (R^*-\epsilon)\sin\left(\frac{\pi}{2}-\epsilon\right)<0.
\end{align*}
This contradiction asserts that $\theta_\ell$ is bounded above. Similarly, for the second case, we obtain, at $t=t^*_2$,
\begin{align*}
    d_\ell\dot{\theta}_\ell
    =
    \omega_\ell + \lambda R(t_2^*) \sin\left( \Phi(t_2^*) - \Phi^* + \frac{\pi}{2} \right)
    >
    \omega_\ell+\lambda (R^*-\epsilon)\sin\left(\frac{\pi}{2}-\epsilon\right)>0.
\end{align*}
This completes the proof.
\end{proof}

\color{black}

\section{Numerical Simulations}
\label{sec6}

In this section, we present eight numerical simulation scenarios to verify our analytical framework and demonstrate the equivalence of synchronization states across various network and timescale configurations. To strictly isolate the topological effects of the coupling matrix $\Lambda$ and the timescale heterogeneity $d_j$, we extract the common simulation parameters to a shared baseline.

\subsection{Simulation Setup and Integration Methodology}\label{subsec:numer_setup}
Across all eight scenarios presented in this manuscript, we simulate a network of $N=100$ coupled oscillators, with the same set of initial phases and natural frequencies. The latter two are sampled from certain distributions, to be detailed below, and then fixed for all scenarios.

To facilitate the description of sampling mechanisms used in generating the parameters, let us first introduce two notations for the type of probability distributions used. The uniform distribution over an interval $\mathsf{I}\subset\mathbb{R}$ is denoted $\mathcal{U}(\mathsf{I})$. The Gaussian (normal) distribution with mean $\mu\in\mathbb{R}$ and standard deviation $\sigma>0$ is denoted $\mathcal{N}(\mu,\sigma^2)$. We also use the standard abbreviations ``r.v." for ``random variables", and ``i.i.d." for a shorthand for ``independently and identically distributed". We write $X\sim\mathcal{P}$ as a shorthand for ``r.v. $X$ is sampled from probability distribution $\mathcal{P}$".

The initial phases are fixed to a single realization sampled from $N$ i.i.d. $\mathcal{U}([0,2\pi))$ r.v.s. The natural frequencies are sampled from $N$ i.i.d. $\mathcal{N}(0,0.5^2)$ r.v.s, subsequently centered to ensure a zero mean (i.e., $\sum_{j=1}^N \omega_j = 0$, namely, using the co-rotating frame), then fixed. In particular, the range of the natural frequencies is $\omega_\mathrm{M}-\omega_\mathrm{m} \approx 2.3697$, and the maximum absolute natural frequency is $\omega := \max_{j\in\{1,\ldots,N\}}|\omega_j| \approx 1.2162$. The exact values for initial phases and natural frequencies can be found in the publicly available simulation dataset~\cite{lo_simulation_2026}.

Given the nonlinear nature of Kuramoto models and the requirement to capture asymptotic stability, simulations were conducted over an extended time horizon of $T=1000$~units. We utilized MATLAB's \texttt{ode89} solver, an adaptive step-size integrator based on an explicit Runge-Kutta (8,7) pair. To prevent the accumulation of local truncation errors, strict error boundaries were enforced (relative tolerance of $10^{-8}$ and absolute tolerance of $10^{-10}$). 

In the following, based on specific settings of coupling and timescale coefficients, we divide the simulated scenarios into four subsections, each with two cases, the first with all types of synchronization states concurring, the second with all states failing concurrently. In each case, both the short-term ($t\in[0,10]$) and long-term ($t\in[0,T]$) evolutions of oscillator phases $\theta_j(t)$, frequencies $\dot{\theta}_j(t)$, and complex order parameters $Z(t)$ are presented in one figure. In particular, we plot the magnitude $R(t)$ and the phase $\Phi(t)$ of $Z(t)$ on the same sub-figure. Recall that in Def.~\ref{def 3}, we choose $\Phi(t)\in\mathbb{R}$ to be continuous such that $\Phi(0)\in[0,2\pi)$. This means that $\Phi(t)$ is not restricted to be the principal argument of $Z(t)$. This choice aids visualization of $Z(t)$ evolution especially in cases without synchronization, where the principal argument of $Z(t)$ can wrap around $2\pi$ frequently.

Animations of oscillators' phase evolutions around the unit circle in each of the eight scenarios are generated and publicly available in the dataset~\cite{lo_simulation_2026}. In these animations, the complex order parameter (physically, the phase centriod) is represented as a black ``star". For some, these animations might serve as a better visualization of synchronization phenomena or the lack thereof. It is worth noting that within the same dataset, final (i.e., at $t=T$) values of oscillator phases, frequencies, and order parameter are also available.

\subsection{Classical Kuramoto Model}
\label{subsec:numer_classical}

\begin{figure}[!ht]\centering\includegraphics[width=0.95\linewidth]{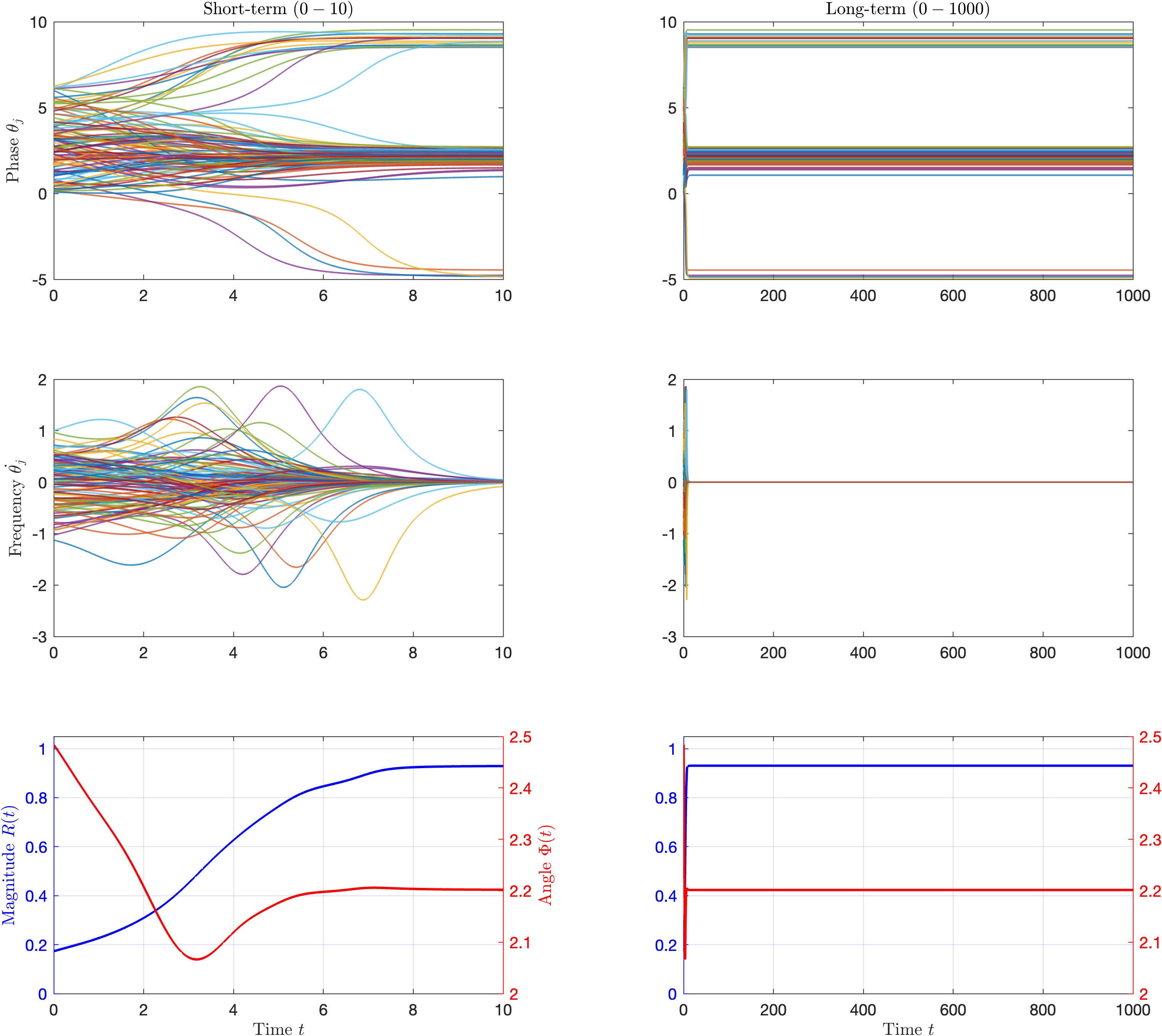}
\caption{Phase, frequency, and order parameter evolution of $N=100$ oscillators in the classical Kuramoto system~\eqref{classical Kuramoto}. The coupling strength is constant with $\lambda=1.44$. 
The maximum absolute final frequency $\max_j|\dot{\theta}_j(T)|\approx 1.8443\times10^{-10}$, demonstrating frequency synchronization. The observed final ($t=T$) order parameter is $Z(T)\approx-0.5488+0.7515i$, with $R(T)\approx 0.9305$ and $\Phi(T)\approx2.2015$. 
\color{black}
}
\label{fig:Classical_FirstOrder_Sync}
\end{figure}

\begin{figure}[!ht]\centering\includegraphics[width=0.95\linewidth]{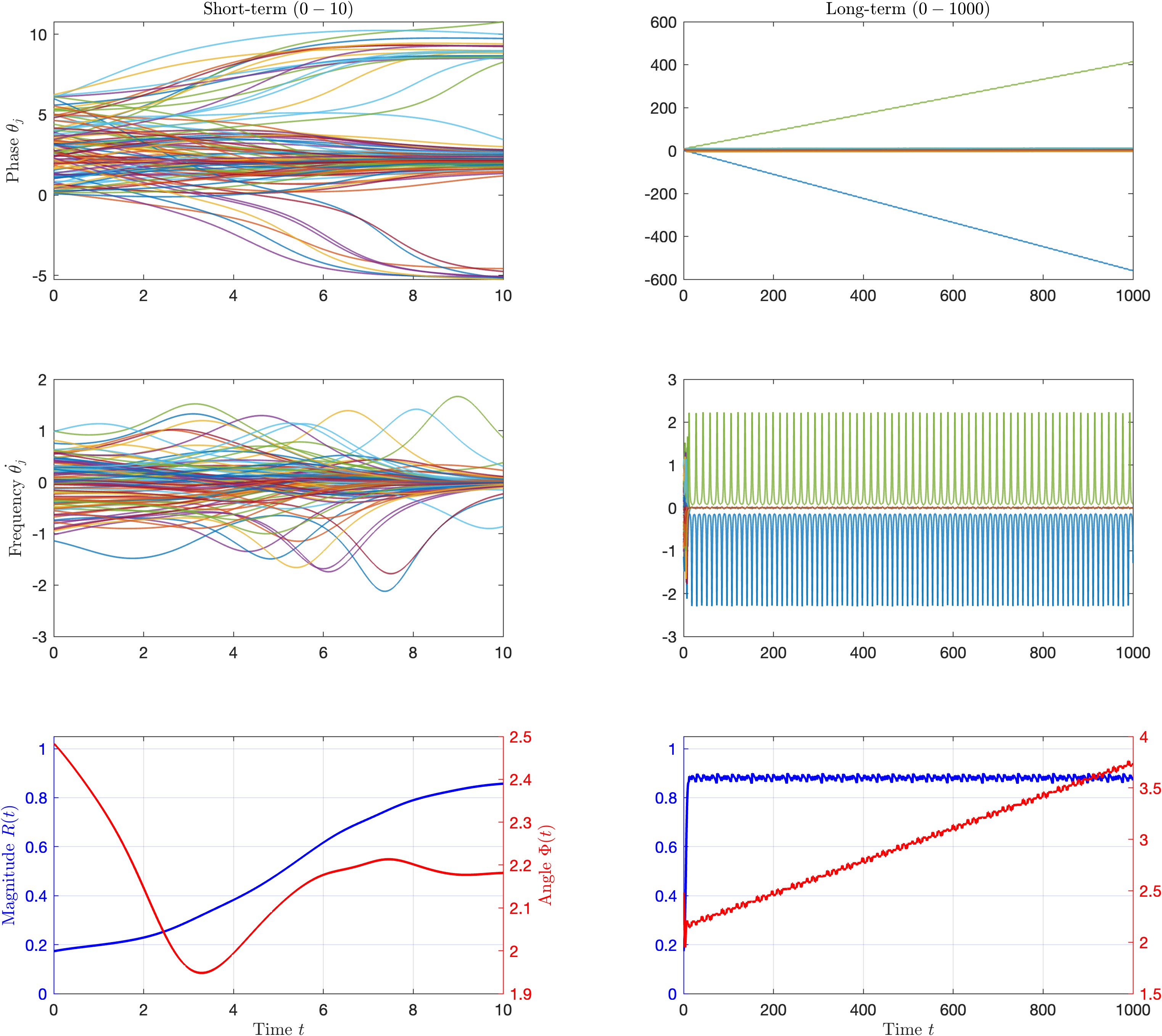}
\caption{Phase, frequency, and order parameter evolution of $N=100$ oscillators in the classical Kuramoto system~\eqref{classical Kuramoto}. The coupling strength is constant with $\lambda=1.22$.}
\label{fig:Classical_FirstOrder_No_Sync}
\end{figure}

We first present numerical simulations for the classical Kuramoto system~\eqref{classical Kuramoto} under constant, all-to-all coupling (i.e., $\lambda_{jk}\equiv\lambda>0$ for all $j\neq k$) and constant timescales ($d_j=1$ for all $j$) with different $\lambda$ values. In Fig.~\ref{fig:Classical_FirstOrder_Sync}, we take $\lambda=1.44$, while in Fig.~\ref{fig:Classical_FirstOrder_No_Sync}, we take $\lambda=1.22$.

The dynamics observed in Fig.~\ref{fig:Classical_FirstOrder_Sync} and Fig.~\ref{fig:Classical_FirstOrder_No_Sync} are perfectly consistent with Theorem~\ref{OP eqiv sync}. Specifically, all modes of synchronization discussed in this paper—(full) phase-locking, frequency synchronization, and order parameter synchronization—occur simultaneously in Fig.~\ref{fig:Classical_FirstOrder_Sync}, whereas none of these modes occur in Fig.~\ref{fig:Classical_FirstOrder_No_Sync}. 

We observe in Fig.~\ref{fig:Classical_FirstOrder_Sync} that $\Phi(t)$ converges roughly to $\Phi(T)\approx 2.2015$, and $R(t)$ converges roughly to $R(T)\approx 0.9305 \geq 0.8446 \approx \omega/\lambda$, satisfying the bounds in \eqref{Phi limit}, so by Def.~\ref{def 3}, this is an OPSS. Moreover, we can verify that the smallest upper bound for $|Z^*|:=\lim_{t\to\infty}R(t)$ given by Corollary~\ref{upper bdd for R}, calculated as $1-\frac{1}{N}+\frac{1}{N}\sqrt{1-\frac{\omega^2}{\lambda^2}}\approx 0.9954$, successfully bounds the empirically observed limit of roughly $0.9305$. 

Fig.~\ref{fig:Classical_FirstOrder_No_Sync} provides a concrete demonstration of the advantage of our newly proposed necessary condition on $\lambda$ for synchronization, as elaborated in the proof of Thm.~\ref{NC for sync}. Evaluating the two bounds in this case reveals that
\begin{align*}
\frac{\omega(N^2+1)}{N^2-N+\sqrt{2 N}}\approx1.2269>\lambda=1.22>1.2088\approx\frac{(\omega_{\mathrm{M}}-\omega_{\mathrm{m}})N}{2\sin(\theta_{\mathrm{opt}})+2(N-2)\sin\left(\frac{\theta_{\mathrm{opt}}}{2}\right)}
,
\end{align*}
providing a concrete example that our newly proposed necessary condition successfully rules out the possibility of synchronization at $\lambda=1.22$, whereas the established necessary condition provided in~\cite{chopra2009exponential} (the rightmost value) fails to rule it out.

An interesting observation in Fig.~\ref{fig:Classical_FirstOrder_No_Sync} is that except for two oscillators, the rest $98$ oscillators appears to achieve partial synchronization. We believe this explains why the order parameter magnitude $R(t)$ appears to stabilize around a constant value with apparent ongoing small oscillatory deviations. Moreover, $\Phi(t)$ appears to be increasing without bound as $t$ increases, revealing that the order parameter phase keeps wrapping around the unit circle. We leave the possibility of characterizing partial synchronization phenomena in Kuramoto models using the newly proposed concept of order parameter synchronization as future work.

\subsection{Heterogeneous Coupling with Constant Timescales}
\label{subsec:numeric_non-unif_coup}
We now extend our simulations to the generalized Kuramoto system~\eqref{G Kuramoto} under heterogeneous, symmetric coupling, while maintaining constant timescales ($d_j=1$ for all $j$). In Fig.~\ref{fig:Nonuniform_FirstOrder_Sync}, the coupling coefficients $\lambda_{jk}$ are drawn i.i.d. from $\mathcal{N}(2,0.5^2)$ for all $j>k$, and in Fig.~\ref{fig:Nonuniform_FirstOrder_NoSync}, we sample $\lambda_{jk}\sim\mathcal{N}(0,0.5^2)$ for all $j>k$ in an i.i.d. manner. Notably, in the second scenario (Fig.~\ref{fig:Nonuniform_FirstOrder_NoSync}), the sampled coupling coefficients are sign-mixed, representing localized antagonistic interactions.

\begin{figure}[!ht]\centering\includegraphics[width=0.95\linewidth]{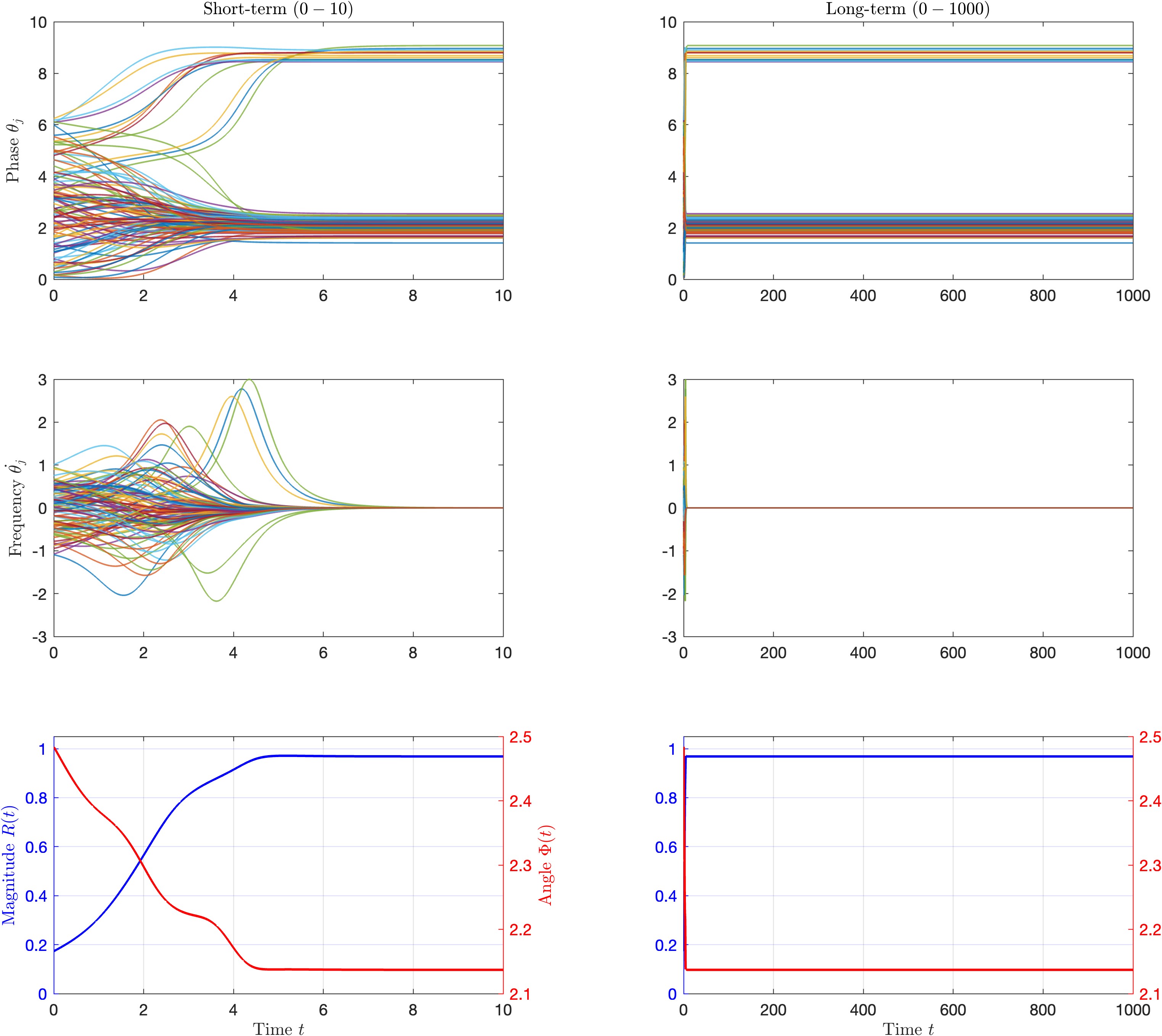}
\caption{Phase, frequency, and order parameter evolution of $N=100$ oscillators in the
general first-order Kuramoto system~\eqref{G Kuramoto} with uniform timescales ($d_j=1$) but heterogeneous symmetric coupling. We sampled $\lambda_{jk}\sim\mathcal{N}(2,0.5^2)$, for all $j>k$, in an i.i.d. manner. The maximum absolute final frequency $\max_j|\dot{\theta}_j(T)|\approx 5.0152\times10^{-10}$, demonstrating frequency synchronization. The observed final ($t=T$) order parameter is $Z(T)\approx-0.5196+0.8173i$, with $R(T)\approx 0.9685$ and $\Phi(T)\approx2.1370$. 
\color{black}
}
\label{fig:Nonuniform_FirstOrder_Sync}
\end{figure}

\begin{figure}[!ht]\centering\includegraphics[width=0.95\linewidth]{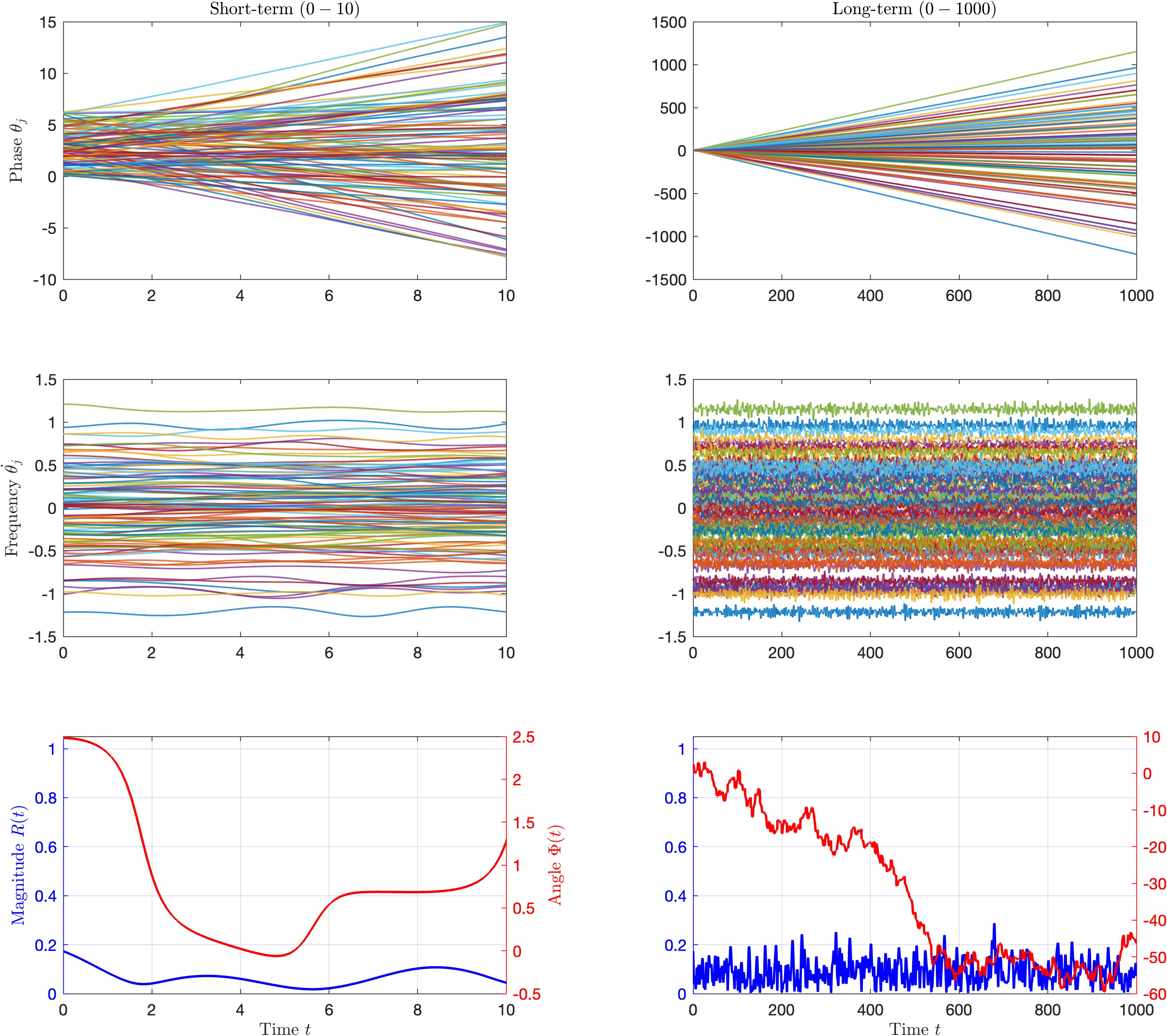}
\caption{Phase, frequency, and order parameter evolution of $N=100$ oscillators in the general first-order Kuramoto system~\eqref{G Kuramoto} with uniform timescales ($d_j=1$) but heterogeneous symmetric coupling. We sampled $\lambda_{jk}\sim\mathcal{N}(0,0.5^2)$, for all $j>k$, in an i.i.d. manner. The realized $\lambda_{jk}$ are sign-mixed, representing antogonistic coupling.}
\label{fig:Nonuniform_FirstOrder_NoSync}\end{figure}

The dynamics in Fig.~\ref{fig:Nonuniform_FirstOrder_Sync} and Fig.~\ref{fig:Nonuniform_FirstOrder_NoSync} remain consistent with Theorem~\ref{main 1}: the modes of phase-locking and frequency synchronization either co-occur (Fig.~\ref{fig:Nonuniform_FirstOrder_Sync}) or concurrently fail (Fig.~\ref{fig:Nonuniform_FirstOrder_NoSync}). 

We remark that, following the proof of Theorem~\ref{OP eqiv sync}, one can analytically show that OPSS is in fact implied by any mode in Definition~\ref{def 2}, even for the case of generic (heterogeneous) symmetric couplings, assuming uniform timescales. However, it remains an open question whether the converse statement—that OPSS guarantees the other modes—holds under the same assumptions on coupling and timescale coefficients. Extensive numerical searches have thus far failed to produce a counter-example to this converse: in all our simulated instances where OPSS occurred, all other modes of synchronization also emerged. We thus conjecture that the converse is true, and leave its proof as future work.

\subsection{Heterogeneous Coupling and Non-Constant Timescales}
\label{subsec:numeric_non-unif_coup_and_dj}

\begin{figure}[!ht]\centering\includegraphics[width=0.95\linewidth]{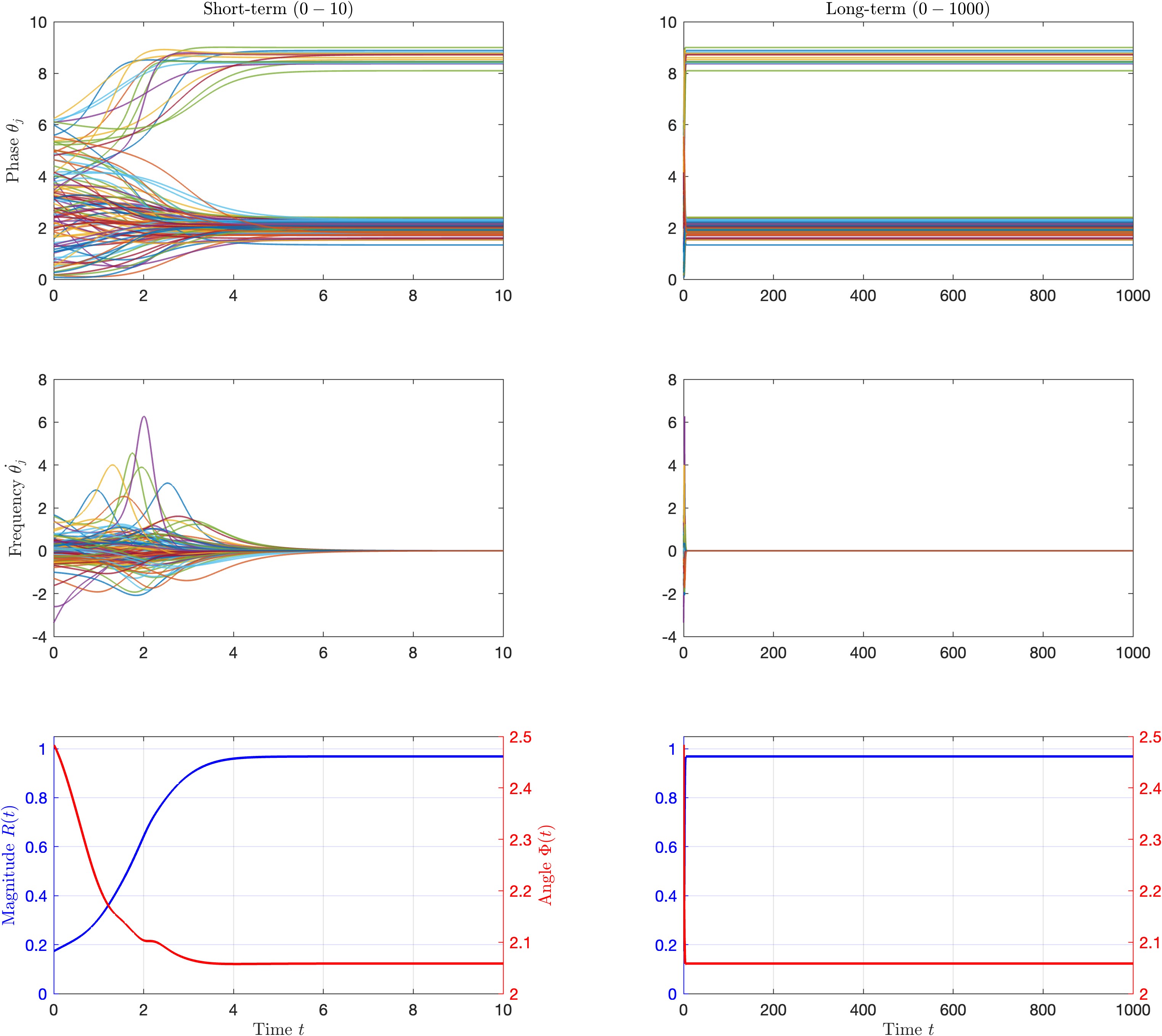}
\caption{
Phase, frequency, and order parameter evolution of $N=100$ oscillators in the general first-order Kuramoto system~\eqref{G Kuramoto}. Timescale coefficients are heterogeneous, sampled i.i.d. as $d_j \sim \mathcal{U}(0.25,1.75)$. Coupling coefficients are chosen identically to those of Fig.~\ref{fig:Nonuniform_FirstOrder_Sync}, where they were sampled i.i.d. as $\lambda_{jk}\sim\mathcal{N}(2,0.5^2)$ for all $j>k$. The maximum absolute final frequency $\max_j|\dot{\theta}_j(T)|\approx 2.2564\times10^{-9}$, demonstrating frequency synchronization. The observed final ($t=T$) order parameter is $Z(T)\approx-0.4539+0.8555i$, with $R(T)\approx 0.9685$ and $\Phi(T)\approx2.0586$. 
\color{black}
}
\label{fig:Nonuniform_Unequal_dj_FirstOrder_Sync}
\end{figure}

\begin{figure}[!ht]\centering\includegraphics[width=0.95\linewidth]{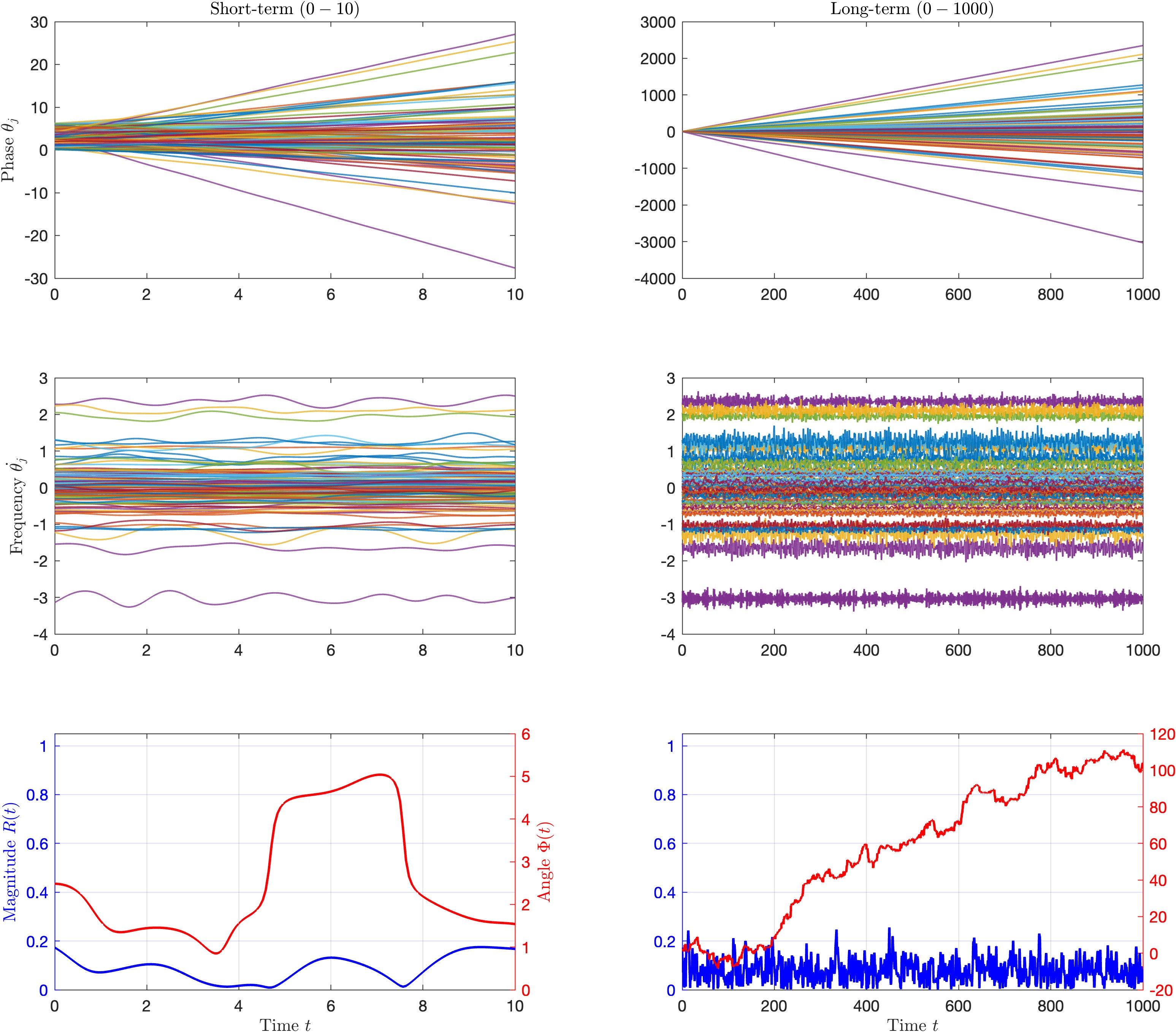}
\caption{
Phase, frequency, and order parameter evolution of $N=100$ oscillators in the general first-order Kuramoto system~\eqref{G Kuramoto}. Timescale coefficients are chosen identically to those of Fig.~\ref{fig:Nonuniform_Unequal_dj_FirstOrder_Sync}, where they were sampled i.i.d. as $d_j \sim \mathcal{U}(0.25,1.75)$. Coupling coefficients are chosen identically to those of Fig.~\ref{fig:Nonuniform_FirstOrder_NoSync}, where they were sampled i.i.d. as $\lambda_{jk}\sim\mathcal{N}(0,0.5^2)$ for all $j>k$.
}
\label{fig:Nonuniform_Unequal_dj_FirstOrder_NoSync}
\end{figure}

Finally, we investigate the ``fully heterogeneous" Kuramoto system~\eqref{G Kuramoto}, i.e., with generic heterogeneous symmetric coupling and non-constant timescales. We present the system dynamics where the timescale coefficients are sampled i.i.d. as $d_j \sim \mathcal{U}(0.25, 1.75)$, and then fixed. In Fig.~\ref{fig:Nonuniform_Unequal_dj_FirstOrder_Sync}, we adopt the same realized samples of the positive-mean heterogeneous coupling as in Fig.~\ref{fig:Nonuniform_FirstOrder_Sync}, where $\lambda_{jk}\sim\mathcal{N}(2,0.5^2)$ for $j>k$. Simiarly, in Fig.~\ref{fig:Nonuniform_Unequal_dj_FirstOrder_NoSync}, we adopt the same realized samples of the zero-mean heterogeneous coupling as in Fig.~\ref{fig:Nonuniform_FirstOrder_NoSync}, where $\lambda_{jk}\sim\mathcal{N}(0,0.5^2)$ for $j>k$. Again, the results in Fig.~\ref{fig:Nonuniform_Unequal_dj_FirstOrder_Sync} and Fig.~\ref{fig:Nonuniform_FirstOrder_NoSync} is consistent with Thm.~\ref{main 1}.

Because of the deliberate choice of parameters, it is interesting to compare Fig.~\ref{fig:Nonuniform_FirstOrder_Sync} with Fig.~\ref{fig:Nonuniform_Unequal_dj_FirstOrder_Sync}, and similarly Fig.~\ref{fig:Nonuniform_FirstOrder_NoSync} with Fig.~\ref{fig:Nonuniform_Unequal_dj_FirstOrder_NoSync}, to observe the effect of non-constant timescales with identical heterogeneous coupling coefficients. In both comparisons, it seems that the effect of switching from a constant, unit timescale to a non-constant but uniformly distributed timescales around the same mean (unity) does not induce significant deviations in the system dynamics. This leads us to investigate the next two scenarios, where we deliberately introduced timescales that differ by at most $4$ orders of magnitude (in base $10$).

\subsection{Heterogeneous Coupling and Deterministic Timescales Spanning Several Orders of Magnitude}
\label{subsec:numeric_multiscale}

\begin{figure}[!ht]
    \centering
    \includegraphics[width=0.95\linewidth]{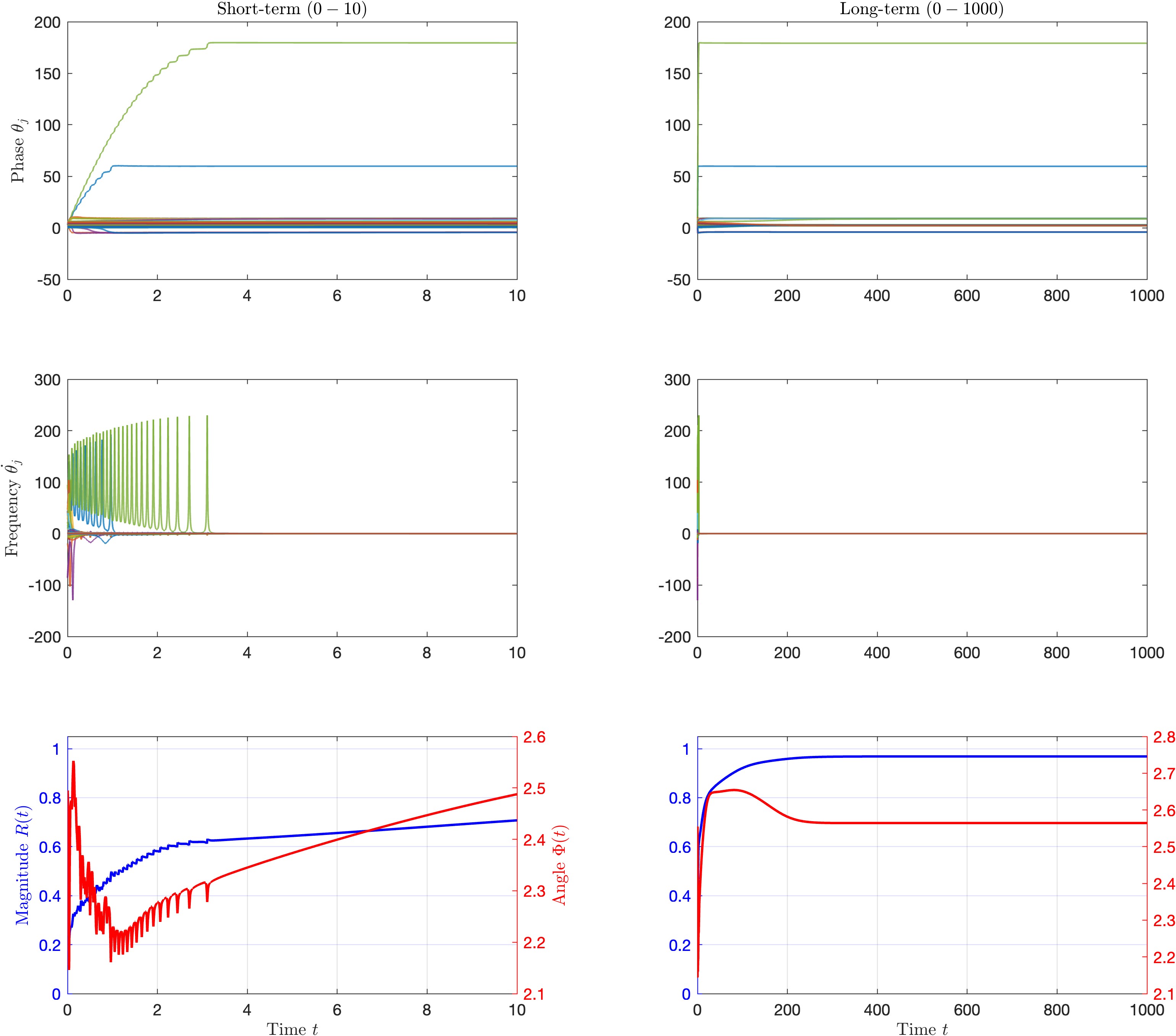}
    \caption{
Phase, frequency, and order parameter evolution of $N=100$ oscillators in the general first-order Kuramoto system~\eqref{G Kuramoto}. Timescale coefficients $d_j$ are deterministic and multi-scale, spanning $\{10^{-2}, 10^{-1}, 10^0, 10^1, 10^2\}$. Coupling coefficients are chosen identically to those of Fig.~\ref{fig:Nonuniform_FirstOrder_Sync} and Fig.~\ref{fig:Nonuniform_Unequal_dj_FirstOrder_Sync}, where they were sampled i.i.d. as $\lambda_{jk}\sim\mathcal{N}(2,0.5^2)$ for all $j>k$. The maximum absolute final frequency $\max_j|\dot{\theta}_j(T)|\approx 5.6374\times10^{-7}$, demonstrating frequency synchronization. The observed final ($t=T$) order parameter is $Z(T)\approx-0.8116+0.5285i$, with $R(T)\approx 0.9685$ and $\Phi(T)\approx2.5644$. 
\color{black}
}
    \label{fig:Multiscale_Sync}
\end{figure}

\begin{figure}[!ht]
    \centering
    \includegraphics[width=0.95\linewidth]{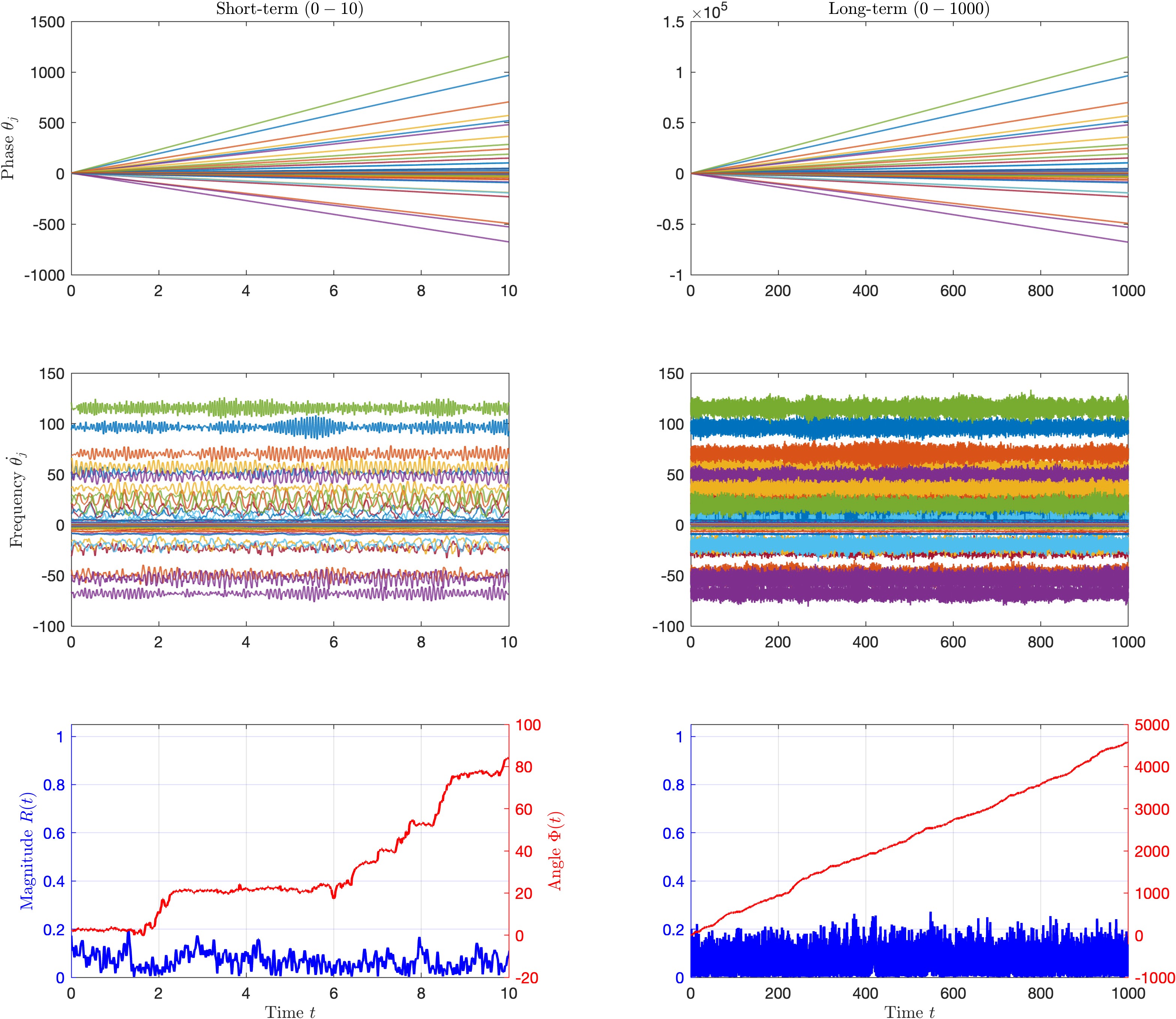}
    \caption{
Phase, frequency, and order parameter evolution of $N=100$ oscillators in the general first-order Kuramoto system~\eqref{G Kuramoto}. Timescale coefficients $d_j$ are deterministic and multi-scale, spanning $\{10^{-2}, 10^{-1}, 10^0, 10^1, 10^2\}$. Coupling coefficients are chosen identically to those of Fig.~\ref{fig:Nonuniform_FirstOrder_NoSync} and Fig.~\ref{fig:Nonuniform_Unequal_dj_FirstOrder_NoSync}, where they were sampled i.i.d. as $\lambda_{jk}\sim\mathcal{N}(0,0.5^2)$ for all $j>k$.
}
    \label{fig:Multiscale_NoSync}
\end{figure}

To test the limits of our proposed equivalence theorem among various synchronization modes under extreme dynamic heterogeneity, we construct a final pair of scenarios where the timescale coefficients $d_j$ span multiple orders of magnitude. Specifically, the network is partitioned into five equal-sized tiers (20 oscillators each), with each tier assigned a deterministic timescale coefficient from the set $\{10^{-2}, 10^{-1}, 10^0, 10^1, 10^2\}$. 

Figure~\ref{fig:Multiscale_Sync} illustrates the system dynamics when this highly heterogenous timescale setup is paired with the heterogeneous coupling coefficients identical to that of Fig.~\ref{fig:Nonuniform_FirstOrder_Sync} and Fig.~\ref{fig:Nonuniform_Unequal_dj_FirstOrder_Sync}, namely, the positive-mean heterogeneous coupling $\lambda_{jk}\sim\mathcal{N}(2,0.5^2)$ for $j>k$. The extreme variance in $d_j$ naturally induces highly complex and hierarchical transient dynamics. Oscillators with small timescale coefficients ($d_j \ll 1$) adjust their phases almost instantaneously, forming rapid local clusters, whereas oscillators with significant inertia ($d_j \gg 1$) dictate the macroscopic settling time of the network. This observation is particular apparent in the accompanying animation in the dataset~\cite{lo_simulation_2026}.

Despite this severe scale separation and the prolonged transient phase, the system ultimately achieves (full) phase-locking and frequency synchronization. We believe this scenario demonstrates the robustness of our proposed equivalence theorem Thm.~\ref{main 1} among wide-ranging timescales. Also, it can be observed that order parameter synchronization is achieved in Fig.~\ref{fig:Multiscale_Sync}.

Figure~\ref{fig:Multiscale_NoSync} presents the identical multi-orders-of-magnitude timescale distribution, paired with mean-zero antagonistic coupling coefficients identical to those of Fig.~\ref{fig:Nonuniform_FirstOrder_NoSync} and Fig.~\ref{fig:Nonuniform_Unequal_dj_FirstOrder_NoSync}, where $\lambda_{jk}\sim\mathcal{N}(0,0.5^2)$ i.i.d. for $j>k$. Under these conditions, the network entirely fails to synchronize across all definitions. The extreme multi-scale nature of the system exacerbates the continuous, incoherent scattering of the phases. The scenario in Fig.~\ref{fig:Multiscale_NoSync} visually suggests that the absence of ``macroscopic" order (i.e., order parameter synchronization) precludes ``microscopic" phase-locking or frequency synchronization. While not covered by Thm.~\ref{main 1}, the contrapositive statement that all other modes of synchronization implies OPSS in generic heterogeneous coupling and timescale coefficients also poses as an interesting conjecture for future work.

\color{black}

\section{Conclusion}
On the fiftieth anniversary of the Kuramoto model, this work clarifies a foundational question in synchronization theory: whether commonly used synchronization criteria represent distinct collective regimes or the same asymptotic dynamical behavior viewed at different levels. For finite-dimensional Kuramoto flows with generic coefficients, we prove that full phase locking, phase locking, and frequency synchronization are dynamically equivalent, independently of the underlying network structure. In the canonical all-to-all setting, we further show that order-parameter synchronization is equivalent to these phase-based notions, providing a rigorous dynamical justification for the Kuramoto order parameter as a macroscopic indicator of synchronization.

Our approach is fully nonlinear and non-perturbative, relying on the global geometry and periodic structure of the phase-difference flow together with a finite-root mechanism. As an application of the unified framework, we also derive a necessary condition that rules out synchronization below a critical coupling threshold and corroborate the theory with numerical experiments. These results offer a coherent dictionary between microscopic phase dynamics and macroscopic coherence and provide a structural basis for extending equivalence principles to higher-order oscillator models, mixed-sign interactions, and broader networked systems.

\section{Appendix}
\subsection{Necessary condition} \label{5_1}
\begin{theorem}  \label{NC for sync} Let us consider the system \eqref{d classical Kuramoto}. There exists a coupling strength
\begin{align}
    \label{eq:critical_coupling_strength}
    \lambda_c
    :=
    \max\left\{
    \frac{\omega(N^2+1)}{N^2-N+\sqrt{2 N}}
    ~,~
    \frac{(\omega_{\mathrm{M}}-\omega_{\mathrm{m}})N}{2\sin(\theta_{\mathrm{opt}})+2(N-2)\sin\left(\frac{\theta_{\mathrm{opt}}}{2}\right)}
    \right\}
\end{align}
such that for $\lambda<\lambda_c$ the solution $\theta(t)$ cannot achieve synchronization in any sense defined in Definition~\ref{def 2} or Definition~\ref{def 3}, where $\omega:=\max_{j\in\{1,\ldots,N\}}|\omega_j|$, $\omega_\mathrm{M}:=\max_{j\in\{1,\ldots,N\}} \omega_j$, $\omega_\mathrm{m}:=\min_{j\in\{1,\ldots,N\}} \omega_j$ and
\begin{align*}
\theta_{\mathrm{opt}}=2\cos^{-1}\left(\frac{-(N-2)+\sqrt{(N-2)^2+32}}{8}\right)
    .
\end{align*}
\end{theorem}

\begin{corollary} \label{upper bdd for R}
 Let $\theta(t)$ be a synchronization state (Definition \ref{def 2} or Definition \ref{def 3}) of the system \eqref{d classical Kuramoto}. Let $\lambda>\omega=\max_{j\in\{1,\ldots,N\}}|\omega_j|>0$. If 
 \begin{align*}
     1>R_0\geq 1-\frac{1}{N}+\frac{1}{N}\sqrt{1-\frac{\omega^2}{\lambda^2}},
 \end{align*}
 then $R_0$ is an upper bound of $\lim_{t\rightarrow\infty}R(t)$.
\end{corollary}
\begin{proof}
    Recall Definition \ref{def 3}. We obtain
    \begin{align*}
        R=Re^{\im(\Phi-\Phi)}=\frac{1}{N}\sum_{j=1}^N e^{\im(\theta_j-\Phi)}.
    \end{align*}
    The real part yields 
    \begin{align*}
        R=\frac{1}{N}\sum_{j=1}^N \cos(\theta_j-\Phi).
    \end{align*}
    If $\lim_{t\rightarrow\infty} R(t)>R_0$, then $\cos(\theta_j-\Phi)$ is at least $N R_0-(N-1)$ for all $j\in\{1,\ldots,N\}$ for sufficiently large $t$. This implies that, for all $j\in\{1,\ldots,N\}$,
    \begin{align*}
    \sin^2(\theta_j-\Phi)=1-\cos^2(\theta_j-\Phi)<1-(N R_0-(N-1))^2,  
    \end{align*}
    for sufficiently large $t$. By assumption, we obtain
    \begin{align*}
        \frac{\omega^2_k}{\lambda^2}=\frac{\omega^2}{\lambda^2}>1-(NR_0-(N-1))^2>\sin^2(\theta_k-\Phi)
    \end{align*}
    for some $k\in\{1,\ldots,N\}$. This means that 
    $|\dot{\theta}_k|$ is bounded below by a nonzero constant 
    \begin{align*}
    |\omega_k|-\lambda\sqrt{1-(NR_0-(N-1))^2}.    
    \end{align*}
    By Theorem \ref{main 1}, $\theta(t)$ is not a synchronization state. However, this contradicts the assumption. The proof is complete.  
\end{proof}

Next, by incorporating \cite{chopra2009exponential}, we can establish a sharper necessary condition for synchronization. 
\begin{proof} [Proof of Theorem~\ref{NC for sync}]
    The second term on the right-hand side of~\eqref{eq:critical_coupling_strength} is quoted directly from~\cite[Theorem 2.1]{chopra2009exponential}, so we only prove the first term. 
    
    Suppose that $\theta(t)$ achieves synchronization in any sense defined in Definition~\ref{def 2} or Definition~\ref{def 3}. A closer look at the proof of Theorem~\ref{OP eqiv sync} reveals that any mode of synchronization defined in Definition~\ref{def 2} implies OP synchronization (i.e., Definition~\ref{def 3}) without any assumptions on the natural frequencies. Combining \eqref{Phi limit} and Corollary \ref{upper bdd for R}, we have
    \begin{align*}
        1-\frac{1}{N}+\frac{1}{N}\sqrt{1-\left(\frac{\omega}{\lambda}\right)^2}
        \geq
        \frac{\omega}{\lambda}
        .
    \end{align*}
    A straightforward calculation reveals that this inequality is equivalent to
    \begin{align*}
        \frac{\omega}{\lambda}
        \leq
        \frac{N^2-N+\sqrt{2N}}{N^2+1}
        .
    \end{align*}
    Therefore, the condition $\lambda<\lambda_c$ implies
    \begin{align*}
        \lambda
        <
        \frac{\omega(N^2+1)}{N^2-N+\sqrt{2 N}}
        ,
    \end{align*}
    which is a contradiction.
     This completes the proof of Theorem \ref{NC for sync}.
\end{proof}
We pause to remark that the two terms in \eqref{eq:critical_coupling_strength} do not have a direct ordering relationship. For example, when $N=2$, 
\begin{align*}
    \frac{5}{4}\omega
    =
    \left.\frac{\omega(N^2+1)}{N^2-N+\sqrt{2N}}\right|_{N=2}
    \leq
    \left.\frac{(\omega_{\mathrm{M}}-\omega_{\mathrm{m}})N}{2\sin(\theta_{\mathrm{opt}})+2(N-2)\sin\left(\frac{\theta_{\mathrm{opt}}}{2}\right)}\right|_{N=2}
    =2\omega.
\end{align*}
However, in the large $N$ limit, we obtain
\begin{align*}
    \omega
    =
    \left.\frac{\omega(N^2+1)}{N^2-N+\sqrt{2N}}\right|_{N=\infty}
    \geq
    \left.\frac{(\omega_{\mathrm{M}}-\omega_{\mathrm{m}})N}{2\sin(\theta_{\mathrm{opt}})+2(N-2)\sin\left(\frac{\theta_{\mathrm{opt}}}{2}\right)}\right|_{N=\infty}
    =\frac{(\omega_{\mathrm{M}}-\omega_{\mathrm{m}})}{2}.
\end{align*}

\backmatter



\section*{Declarations}
The authors declare that they have no conflict of interest. Ethical approval is not applicable.

\section*{Data Availability Statement}

The data that support the findings of this study are openly available at the following URL/DOI: \url{https://doi.org/10.5281/zenodo.19020534}. 

This repository contains the structured parameter datasets (in .csv format), as well as descriptions of the solver and platform used to generate the simulations in this paper. It also contains the high-resolution $(1200~\text{pixel}\times1200~\text{pixel})$ phase evolution animations (.mp4) corresponding to each simulation scenario in the manuscript.
\color{black}

\bibliography{sn-bibliography}

\end{document}